%% file: main.tex
\documentclass[hidelinks,onefignum,onetabnum]{siamart220329}

\input{shared}
\input{definitions}

\ifpdf
\hypersetup{
  pdftitle={A Neural Network Kernel Decomposition for Learning Multiple Steady States in Parameterized Dynamical Systems},
  pdfauthor={Yimeng Zhang, Alexander Cloninger, Bo Li, Xiaochuan Tian}
}
\fi

\begin{document}

\maketitle

% REQUIRED
\begin{abstract}
We develop a data-driven machine learning approach to identifying parameters
with steady-state solutions, locating such solutions, and determining their linear stability
for systems of ordinary differential equations and dynamical systems with parameters. 
Our approach first constructs target functions for these tasks, then designs a parameter-solution neural network (PSNN) that couples a parameter neural network and a solution neural network to approximate the target functions. 
We further develop efficient algorithms to train the PSNN and locate steady-state solutions.
An approximation theory for the target functions with PSNN is developed
based on kernel decomposition.
%Our approach begins with the construction of target functions that can be used to identify parameters with steady-state solution and the linear stability of such solutions. We design a parameter-solution neural network (PSNN) that couples a parameter neural network and a solution neural network to approximate the target function, and develop efficient algorithms to train the PSNN and to locate steady-state solutions. We also present a theory of approximation of the target function by our PSNN based on the neural network kernel decomposition.
Numerical results are reported to show that our approach is robust in finding solutions, identifying phase boundaries, and classifying solution stability across parameter regions.
%identifying the phase boundaries separating different regions in the parameter space corresponding to no solution or different numbers of solutions and in classifying the stability of solutions. 
These numerical results also 
validate our analysis. While this study focuses on steady states of parameterized dynamical systems, our approach is equation-free and is applicable generally to finding solutions for parameterized nonlinear systems of algebraic equations.
%\textcolor{red}{Since our approach does not rely on dynamics, it is also applicable to %finding solutions for parameterized nonlinear systems of algebraic equations, although the %primary focus in this study centers on steady states of parameterized dynamical systems.} 
Some potential improvements and future work are discussed. 
\end{abstract}

% REQUIRED

% REQUIRED

%%%%%%%%%%%%%%%%%%%%%%%%%%%Section 1 %%%%%%%%%%%%%%%%%%%%%%%%%%

\section{Introduction}
\label{sec:Introduction}

We develop a machine learning approach to identifying the parameters for which the following 
general parameterized system of autonomous ordinary differential equations (ODEs)
has steady-state solutions (or, equivalently, steady states, or equilibrium solutions):  
\begin{equation}
    \label{ODE00}
\frac{du_i}{dt} = G_i (u_1, \dots, u_n; \theta_1, \dots, \theta_m),  \qquad i = 1, \dots, n, 
\end{equation}
where all $\theta_1, \dots, \theta_m \in \R$ are parameters and 
all $G_1, \dots, G_n$ are given functions. For those parameters with steady-state solutions, we further
determine the linear stability of such solutions.

Parameterized ODEs and dynamical systems such as the system of equations (\ref{ODE00}) 
are powerful and frequently used models for complex systems in social, physical, and biological sciences. 
Properties of solutions $u_i = u_i(t) $ $(i = 1, \dots, n)$ of such equations, such as 
their existence, stability, and long-time behavior, depend sensitively on the parameters, 
and bifurcations can occur when parameters are varied
\cite{Hirsch2012,Perko2001,Strogatz2015}. 
Yet, determining the exact values of the parameters with distinguished 
solution properties is extremely challenging, often due to the limited data from archiving, 
experiment, and computer simulations. 
For instance, ODEs (\ref{ODE00}) serve as main mathematical models of ecology and evolution
\cite{Hastings_PopBiol1996,Turchin_PopDyn2003}. For an ecological system, many 
different species interact with each other and also through 
the competition for resources. Their populations
are solutions of ODEs of the form (\ref{ODE00}) with parameters being the interaction 
strength, rates of consumption, etc. As such a system can be very large with many species, 
there are many different types of parameters, difficult to collect or measure.
Another common example of application of ODEs (\ref{ODE00}) is chemical reaction \cite{Hjortso_ChemReac2010}, particularly gene regulatory networks in system biology \cite{Polynikis2009}. 
Concentrations of many different chemical or biological species are modeled
by ODEs with parameters such as reaction rates that are often difficult
to measure experimentally. 
It is noted that many complex spatio-temporal patterns, such as Turing patterns, arise from 
spatio-temporal perturbations
of steady-state solutions of ODE systems reduced from reaction-diffusion
systems modeling chemical reactions \cite{Turing1952,Walgraef1997}. 
Examples of such reaction-diffusion systems include 
the Gray--Scott and Gierer--Menhardt models 
\cite{GiererMeinhardt1972,GrayScott83,GrayScott84,GrayScott85,Pearson1993,Walgraef1997}.

Computing multiple steady states of a nonlinear system is a complex challenge that has long driven the development of various numerical methods.
These include variational and non-variational methods \cite{allgower2012numerical,deuflhard2005newton,farrell2015deflation,lo2012robust,xu1994novel,yin2020construction,zhou2017solving}. 
 More recently, neural network-based techniques have emerged for learning multiple solutions of nonlinear systems \cite{huang2022hompinns,zheng2023hompinns}. 
Additionally, various computational methods have been developed for bifurcation studies in parameterized systems \cite{hao2022learn,pichi2023artificial,pichi2020reduced}.   
When dealing with parameterized nonlinear systems, the aforementioned methods require repeated computations across various parameter values, resulting in significant computational costs for identifying steady states.
To address this challenge, we propose a data-driven machine learning approach that efficiently learns multiple solutions of parameterized systems over a broad parameter range. Existing methods can be used to generate training data for our algorithm. This approach is equation-free and broadly applicable to parameterized nonlinear systems.
We demonstrate its effectiveness in studying the steady states of parameterized dynamical systems, including locating solutions, predicting stability, and generating phase diagrams to identify parameter regions with multiple solutions.

To achieve our objectives, we develop a new neural network framework named the {\it Parameter-Solution Neural Network} (PSNN).
The PSNN architecture, as illustrated in \Cref{fig:NN}, comprises two crucial subnetworks: the parameter network and the solution network. The subnetworks are interconnected 
 by the inner product of their output vectors, which are specifically designed to be of the same dimension. 
 %%\BL{I revised this next sentence. Changed "designed" to "constructed", 
 %%removed $\in [0, 1]$ as we use a scaled Sigmoid, changed "funciton for" to "function of", 
 %%and also changed "solution pairto "solution-parameter pair.}
 The network is constructed
 to approximate a target function $\Phi = \Phi(U, \Theta)$
 of $U=(u_1, u_2, \cdots, u_n)$ and $\Theta=(\theta_1,\theta_2, \cdots, \theta_m)$ that indicates the probability of $(U, \Theta)$ being a solution-parameter pair. 
 The key theoretical development in our work is a universal approximation theorem with error bound estimates for PSNN to approximate a given function $\Phi(U, \Theta)$ that exhibits different degrees of smoothness in its two arguments. In particular, by our design of the target function $\Phi$ that will be introduced shortly is smooth in $U$ and only piecewise smooth in $\Theta$, which enables us to model phase transitions. 
 %%\BL{I revised the first half of this sentence.}
 For the approximation theorem of the convectional fully connected neural networks, we refer the readers to \cite{petersen2018optimal,yarotsky2017error}.
 Of particular note is that the approximation theory for piecewise smooth functions presented in \cite{petersen2018optimal} by fully connected networks is a crucial stepping stone of our new approximation theory developed for $\Phi$ by PSNN.
 The key to our approximation theory is a spectral decomposition of $\Phi(U,\Theta)$ with rate estimates for the decay of eigenvalues. 
 %%\BL{changed "rates estimate of eigenvalues decays" to rate estimates 
 %%for the decay of eigenvalues".}
This is also reminiscent of principal component analysis (PCA) in statistics, singular value decomposition (SVD) in linear algebra, and Karhunen–Lo\`eve theorem \cite{schwab2006karhunen} in the realm of stochastic processes. By utilizing the spectral decomposition and classical approximation theory for fully connected networks, we can attain a new approximation theory for PSNN. 
%%(need to cite twin network? other potential applications?)

\begin{figure}[htp]
   \begin{tikzpicture}[x=1.6cm,y=1cm]
      \readlist\Nnod{3,3,3,2} % array of number of nodes per layer
  \readlist\Nstr{m,k,N} % array of string number of nodes per layer
  \readlist\Cstr{\theta, h,\Psi} % array of coefficient symbol per layer
  \def\yshift{0.55} % shift last node for dots
  
  % LOOP over LAYERS
  \foreachitem \N \in \Nnod{
    \def\lay{\Ncnt} % alias of index of current layer
    \pgfmathsetmacro\prev{int(\Ncnt-1)} % number of previous layer
    \foreach \i [evaluate={\c=int(\i==\N); \y=\N/2-\i-\c*\yshift;
                 \x=\lay; \n=\nstyle;
                 \index=(\i<\N?int(\i):"\Nstr[\n]");}] in {1,...,\N}{ % loop over nodes
      % NODES
      \ifnumcomp{\lay}{=}{1}{
         \node[node \n] (N\lay-\i) at (\x,\y) {$\strut\Cstr[\n]_{\index}$};
      }
      
      % CONNECTIONS
      \ifnumcomp{\lay}{=}{2}{ % connect to previous layer
        \node[node \n] (N\lay-\i) at (\x,\y) {$h^{1}_{P,\index}$};
        \foreach \j in {1,...,\Nnod[\prev]}{ % loop over nodes in previous layer
          \draw[white,line width=1.2,shorten >=1] (N\prev-\j) -- (N\lay-\i);
          \draw[connect] (N\prev-\j) -- (N\lay-\i);
        }
        \ifnum \lay=\Nnodlen
          %%\draw[connect] (N\lay-\i) --++ (0.5,0); % arrows out
        \fi
      }{
        %%\draw[connect] (0.5,\y) -- (N\lay-\i); % arrows in
      }
      \ifnumcomp{\lay}{=}{3}{ % connect to previous layer
        \node[node \n] (N\lay-\i) at (\x,\y) {$h^{L_1}_{P,\index}$};
        \foreach \j in {\i}{ % loop over nodes in previous layer
          \path (N\prev-\j) --++ (N\lay-\j) node[midway,scale=1.6] {$\cdots$};
          %%\draw[connect] (N\prev-\j) -- (N\lay-\i);
        }
        \ifnum \lay=\Nnodlen
          %%\draw[connect] (N\lay-\i) --++ (0.5,0); % arrows out
        \fi
      }{
        %\draw[connect] (0.5,\y) -- (N\lay-\i); % arrows in
      }
      \ifnumcomp{\lay}{=}{4}{ % connect to previous layer
         \node[node \n] (N\lay-\i) at (\x,\y) {$\strut\Cstr[\n]_{\index}$};
        \foreach \j in {1,...,\Nnod[\prev]}{ % loop over nodes in previous layer
          \draw[white,line width=1.2,shorten >=1] (N\prev-\j) -- (N\lay-\i);
          \draw[connect] (N\prev-\j) -- (N\lay-\i);
        }
        \ifnum \lay=\Nnodlen
          %%\draw[connect] (N\lay-\i) --++ (0.5,0); % arrows out
        \fi
      }{
        %%\draw[connect] (0.5,\y) -- (N\lay-\i); % arrows in
      }
    }
    \path (N\lay-\N) --++ (0,1+\yshift) node[midway,scale=1.6] {$\vdots$}; % dots
  }
  
  % LABELS
  \node[above=3,align=center,mydarkgreen] at (N1-1.90) {};
   \node[right=0.5cm,align=center,mydarkred] (p-network) at (N2-1.90) {};
  \node[above=2,align=center,mydarkblue] at  (p-network) (p-hiddenlayer){\textbf{Parameter-network}};
  \node[above=3,align=center,mydarkred](p-outvec) at (N\Nnodlen-1.90) {};
  %\node[above=0.1cm,align=center,black] at (p-hiddenlayer){\textbf{Parameter-network}};

  \readlist\Nnod{3,3,3,2} % array of number of nodes per layer
  \readlist\Nstr{n,p,N} % array of string number of nodes per layer
  \readlist\Cstr{u,h^{(\prev)},\Xi} % array of coefficient symbol per layer
  %\def\yshiftt{3.55} % shift last node for dots
  
  % LOOP over LAYERS
  \foreachitem \N \in \Nnod{
    \def\lay{\Ncnt} % alias of index of current layer
    \pgfmathsetmacro\prev{int(\Ncnt-1)} % number of previous layer
    \foreach \i [evaluate={\c=int(\i==\N); \y=\N/2-\i-\c*\yshift;
                 \x=\lay; \n=\nstyle;
                 \index=(\i<\N?int(\i):"\Nstr[\n]");}] in {1,...,\N}{ % loop over nodes
      % NODES

      \ifnumcomp{\lay}{=}{1}{
        \node[node \n] (N\lay-\i) at (\x,\y-4.6) {$\strut\Cstr[\n]_{\index}$};
      }
      
      % CONNECTIONS
      \ifnumcomp{\lay}{=}{2}{ % connect to previous layer
        \node[node \n] (N\lay-\i) at (\x,\y-4.6) {$h^{1}_{S,\index}$};
        \foreach \j in {1,...,\Nnod[\prev]}{ % loop over nodes in previous layer
          \draw[white,line width=1.2,shorten >=1] (N\prev-\j) -- (N\lay-\i);
          \draw[connect] (N\prev-\j) -- (N\lay-\i);
        }
        \ifnum \lay=\Nnodlen
          %%\draw[connect] (N\lay-\i) --++ (0.5,0); % arrows out
        \fi
      }{
        %\draw[connect] (0.5,\y) -- (N\lay-\i); % arrows in
      }
      \ifnumcomp{\lay}{=}{3}{ % connect to previous layer
        \node[node \n] (N\lay-\i) at (\x,\y-4.6) {$h^{L_2}_{S,\index}$};
        \foreach \j in {\i}{ % loop over nodes in previous layer
          \path (N\prev-\j) --++ (N\lay-\j) node[midway,scale=1.6] {$\cdots$};
          %%\draw[connect] (N\prev-\j) -- (N\lay-\i);
        }
        \ifnum \lay=\Nnodlen
          %%\draw[connect] (N\lay-\i) --++ (0.5,0); % arrows out
        \fi
      }{
        %\draw[connect] (0.5,\y) -- (N\lay-\i); % arrows in
      }
      \ifnumcomp{\lay}{=}{4}{ % connect to previous layer
        \node[node \n] (N\lay-\i) at (\x,\y-4.6) {$\strut\Cstr[\n]_{\index}$};
        \foreach \j in {1,...,\Nnod[\prev]}{ % loop over nodes in previous layer
          \draw[white,line width=1.2,shorten >=1] (N\prev-\j) -- (N\lay-\i);
          \draw[connect] (N\prev-\j) -- (N\lay-\i);
        }
        \ifnum \lay=\Nnodlen
          %%\draw[connect] (N\lay-\i) --++ (0.5,0); % arrows out
        \fi
      }{
        %\draw[connect] (0.5,\y) -- (N\lay-\i); % arrows in
      }

    }
    \path (N\lay-\N) --++ (0,1+\yshift) node[midway,scale=1.6] {$\vdots$}; % dots
  }
  
  % LABELS

  \node[above=3,align=center,mydarkgreen] at (N1-1.70) {};
  \node[right=0.5cm,align=center,mydarkred] (s-network) at (N2-1.90) {};
  \node[above=2, align=center,mydarkblue](s-hiddenlayer) at (s-network){\textbf{Solution-network}};
  \node[above=3,align=center,mydarkred] (s-outvec) at (N\Nnodlen-1.90) {};
  %\node[below=4.2cm,align=center,black] at (s-hiddenlayer){\textbf{Solution-network}};

  \node[right=4cm, ellipse, draw=myred, fill = myred!20, minimum width = 1cm, minimum height = 1cm, very thick] (Innerproduct) at  (s-hiddenlayer) {$\hat{\Phi}_{\text{PSNN}} = \Psi \cdot \Xi$};
  \node[below=1.3cm,align=center,black] (p-point)at (p-outvec){};
  \node[below=1.4cm,align=center,black](s-point) at (s-outvec){};
  \draw[connect] (p-point) -- (Innerproduct);
  \draw[connect] (s-point) -- (Innerproduct);
  %%\node[right=2.6cm, circle, draw=black, fill = myred!20, very thick] (phi-hat) at  (Innerproduct) {$\hat{\Phi}_{\rm PSNN}$};
  \node[right=3cm, circle, draw=black, fill = myred!20, very thick] (phi) at  (Innerproduct) {$\Phi_{\rm PSNN}$};
  \draw[connect] (Innerproduct) -- (phi);
  %%\draw[connect] (phi-hat) -- (phi);
  \node[right=2cm](point4sigmoid)at (Innerproduct){};
  \node[above=.4cm](sigmoid)at (point4sigmoid){scaled-sigmoid};

\end{tikzpicture}

\caption{Schematic description of the structure of the parameter-solution neural network
$\Phi_{\rm PSNN}$. 
%%\BL{added $\Phi_{\rm PSNN}.$}
}
\label{fig:NN}
\end{figure}
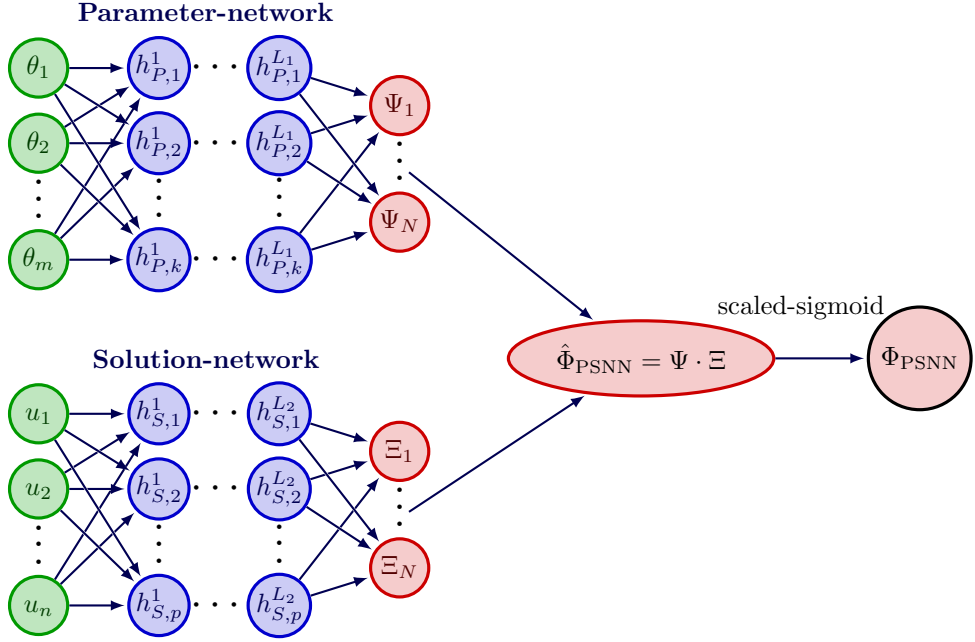

%%%%%%%%%%%%%%%%%%%%%%%
\begin{comment}
\begin{figure}
    \begin{tikzpicture}
        
    \end{tikzpicture}
    \caption{Caption}
    \label{fig:enter-label}
\end{figure}
\end{comment}
%%%%%%%%%%%%%%%%%%%%%%%%%%%%%%

%It is worth noting that the separation of the parameter %and solution networks within the PSNN architecture is 
%particularly advantageous for allowing efficient %approximation of the target function. Within this %framework, the subnetworks can be designed to align with %the specific degrees of freedom needed that are caused 
%by (different) regularities in parameters and solutions. 

\vspace{-4 mm}

Our PSNN is trained using a supervised learning approach, with training data obtained from numerical computations or experiments. After training, we apply post-processing algorithms to identify all potential solutions for given parameters and assess their stability.
We note that the neural network-based approach serves two purposes: finding solutions for a given parameter and identifying parameter regions of specific properties (such as the number of solutions and their stability). 
To demonstrate the necessity of PSNN for these two objectives, we also develop a neural network-free approach based on a mean-shift algorithm, compare its performance with PSNN, and show that the PSNN-based method outperforms it. 
In practice, experimental data may be incomplete.  To exemplify the robustness of our approach, we also introduce strategies for managing incomplete data. 
Application of our methodology to identifying the steady states of the (spatially homogeneous) 
Gray--Scott model is given at the end. Future work will explore broader applications of our approach.\\

%Our PSNN is trained using a ``supervised learning'' approach, where the training data can be obtained either by known solutions or by experiments. Following the training phase, we further develop post-processing algorithms to locate all potential solutions associated with given parameters, assess their stability, and generate phase diagrams that encapsulate the behavior of the ODE system. In practice, experimental data may be incomplete.  To exemplify the broad applicability of our approach, we also introduce effective strategies for managing incomplete data. 
%Finally, we note that while our primary focus in this work revolves around steady states of parameterized dynamical systems, our approach extends to broader applications. 
%In particular, the framework we propose is well-suited for discovering solutions for parameterized nonlinear systems of algebraic equations.  We will address the broader applications of our framework in future work. \\

%In this work, we develop a neural network approach to identifying parameters with 
%steady-state solutions that are often multiply many, locating all such solutions, 
%and determining their stability 
%for the system of ODEs with parameters (\ref{ODE00}). 

\noindent \textbf{Overview of our approach}.
%%\BL{Do we need a subtitle here?}
To facilitate easy reading, we now give an overview of our approach and describe
the key results of our studies.

%\YZ{
%In my understanding, the main purpose of our approach is to find the solutions of parameterized ODE system, especially the multi-solution system, which should be mentioned in the contribution part.
%}

%\BL{I agree that finding multiple solutions is one of the main result we have. I added the word
%``multiple" above. I think there 
%are three objectives: identify parameters with solutions that are often multiple; locate the solutions; 
%and determine their stability. }

%\YZ{
%And the target function is designed to help us learn and locate the solutions, which to my understanding is part of the numerical algorithm
%}

%\BL{It is reasonable to regard the target functions as part of the numerical algorithm. After 
%all, this work is mainly numerical.  Designing such functions is the starting point
%of the method / algorithm. Here, we try to summarize the main components of our method / algorithm, 
%and the main results of this work.}

%\begin{itemize}
%\item[(1)]
(1) {\it Target functions.}
We construct a target function for solution, $\Phi  = \Phi(U, \Theta),$ and a 
target function for linear stability, $\Phi^{\rm s} = \Phi^{\rm s}(U, \Theta), $ of all 
feasible steady-state solution vectors $U = (u_1, \dots, u_n)\in \R^n$ 
and parameter vectors $\Theta = (\theta_1, \dots, \theta_m) \in \R^m$. 
These functions have several distinguished properties that 
can be used, for instance, to determine if for a given parameter $\Theta$ 
the system of ODEs \eqref{ODE00} 
has at least one steady-state solution, and in case so, to locate all such solutions and 
determine their linear stability.
In essence, $\Phi$ is designed to produce values approximately ranging between 0 and 1, representing the probability that an input vector $U$ is a steady-state solution corresponding to an input parameter vector $\Theta$, and $\Phi^s$ further incorporates the stability information. 
The precise definition of $\Phi$ and $\Phi^s$ are found \cref{eq:Phi,eq:Phis}, respectively.
%In the simple case where each parameter vector $\Theta$ corresponds
%to no solution or exactly one solution $U^\Theta$, the target function for solution is given by
%\[
%\Phi(U, \Theta) = \chi_{\Omega_1} (\Theta) \exp{ \left( - \frac{|U - U^{\Theta}|^2}{\delta }\right)},
%\]
%where $\Omega_1$ is the set of all parameter vectors $\Theta $ for each of which
%there exists exactly one steady-state solution to the system \eqref{ODE00}, 
%$\chi_{\Omega_1}$ is the indicator or characteristic function of $\Omega_1$
%($\chi_{\Omega_1}(\Theta)$ equals $ 1$ if $\Theta\in \Om_1$ and equals $0$ otherwise), 
%and $\delta > 0$ is a numerical parameter. 
%The target function for stability $\Phi^{\rm s}$ is defined similarly; see more details in the 
%next section.

%%\item[(2)]
(2) {\it A parameter-solution neural network} (PSNN). 
We design and train a PSNN, denoted $\Phi_{\text{PSNN}}(U, \Theta, \omega),$
as a function of solution vectors $U$ and parameter vectors $\Theta$
to approximate the target function $\Phi(U, \Theta),$ where $\omega = \{\omega_{\rm P}, \omega_{\rm S}\}$ is the set of 
neural network parameters:
\[
\Phi_{\text{PSNN}}(U, \Theta, \omega) = \sigma ( \Phi_{\text{PNN}} (\Theta, \omega_{\text{P}})
\cdot \Phi_{\text{SNN}}(U, \omega_{\text{S}} )),
\]
Here, $\Phi_{\text{PNN}} (\cdot, \omega_{\text{P}})$
is a parameter neural network and $\Phi_{\text{SNN}}(\cdot, \omega_{\text{S}})$
is a solution neural network, both vector-valued with output values in $\R^N$ for some
integer $N \ge 1$, where
$\omega_{\text{P}}$ and $\omega_{\text{S}}$ are the respective sets of neural network parameters. 
The dot means the inner product for two vectors in $\R^N$ and 
the function $\sigma: \R \to \R$ is a scaled sigmoid function to be defined later. 
%We use the ReLu activation function for all of our neural networks.
Figure~\ref{fig:NN} shows the structure of our PSNN.
We similarly define $\Phi^s_{\rm PSNN}$ for learning stability. 
%Once trained, the PSNN outputs a value in $(0, 1)$ approximately
%which is the probability
%that an input vector $U$ is a steady-state solution corresponding to an
%input parameter vector $\Theta.$
%A similar PSNN for the stability of steady-state solutions is constructed and trained.

%%\item[(3)]
(3) {\it Convergence and error bounds.} We prove that, under some realistic 
assumptions on the smoothness of
regions in the parameter space and solution space, any sequence of PSNNs converges
to the target function: 
\[
\Phi_{\text{PSNN}}(U, \Theta, \omega(k)) \to \Phi (U, \Theta) 
\qquad \mbox{as } k \to \infty
\]
for each pair $(U, \Theta),$ where the number of weights in $\omega(k)$ for the 
$k$th PSNN increases to infinity as $k \to \infty.$ Moreover, we provide precise
error bounds for the error between the PSNN and the target function. 
%%\BL{add more result of the analysis - using a few sentences and minimum number of notations.}

%%\item[(4)]
(4) {\it Numerical algorithms.} 
%%\BL{This part needs to be polished more.}
We design an $L^2$-type loss function and implement ADAM, a stochastic optimization 
algorithm, to minimize numerically the lost function and to train our PSNN. 
We also develop a clustering method using the K-means algorithm \cite{Kmeans}
to locate steady-state
solutions after we apply our PSNN to generate data of possible solutions. 

%%\item[(5)]
(5) {\it Numerical results.} 
We test and validate our approach on the spatially homogeneous Gray--Scott model.
Extensive numerical results show that our approach effectively finds steady-state solutions and detects phase boundaries. 
Additionally, our method remains robust even with incomplete data. We demonstrate that the PSNN-based approach outperforms non-neural network methods, such as the mean-shift algorithm \cite{Mean-shift}. 
%We test and validate our approach on the Gray--Scott model
%which is a system of two equations of two unknow functions with two parameters for which analytical formulas
%are available for steady-state solutions and their stability. Our extensive numerical results
%show that our approach can determine parameters with steady-state solutions and 
%detect the boundaries that separate different 
%parameter regions with no solution and with multiple solutions, respectively. Even 
%with the PSNN trained only using incomplete set of training data, 
%our approach is still able to 
%locate solutions. Our clustering method based on PSNN has proven to be more effective compared to other approaches that do not utilize trained neural networks, such as the mean-shift algorithm \cite{Mean-shift}. 
\\

The rest of the paper is organized as follows: In section~\ref{sec:Theory}, 
we formulate our problem and define our target functions, construct our 
neural networks, and provide detailed numerical algorithms for training and using
these networks. In section~\ref{sec:Analysis}, we present a 
convergence analysis and error estimates for our PSNNs. 
In section~\ref{sec:Results}, we report our numerical results to show how our 
PSNNs work and also to validate our analysis. Finally, in section~\ref{sec:Conclusions}, we draw conclusions and discuss some potential improvements and future work. 
Appendix collects an algorithm of the mean-shift method that is used
only for comparison.

%%%%%%%%%%%%%%%%%%%%%%%%%%%Section 2 %%%%%%%%%%%%%%%%%%%%%%%%%%

\section{Parameter-Solution Neural Networks and Numerical Algorithms}
\label{sec:Theory}

%%\BL{I removed (PSNN) in the title - usually abbreviations are not defined in a title.}

%%%%%%%%%%%%%%%%%%%%%%%%%%%%%%%%%%%%%%%%%%%%%%
%%\section{Method}
\subsection{Problem formulation and assumptions}
\label{subsec:formulation}

%%\BL{How about changing the title to ``The target functions"?  Or "Problem formulation and 
%%the target functions"? Or otherwise: Definitions and assumptions?}

Let $n \ge 1$ and $m \ge 1$ be integers and $G_1, \dots, G_n$ be
%%continuously differentiable \XT{Where is this condition used?} 
functions defined on some open subset of $\R^{n \times m}.$ 
We consider the autonomous system \eqref{ODE00} of $n$ ordinary differential equations (ODEs)
for $n$ unknown functions $u_i = u_i (t)$ $(i = 1, \dots, n)$ with $m$ 
parameters $\theta_i \in {\mathbb R}$ $(j = 1, \dots m).$
%%We are interested in identifying parameters $\theta_1, \dots, \theta_m $ for which 
%%the system of ODEs \eqref{ODE00} has at least one steady-state solution.
%%In case such a solution exists, we would
%%like to find all the solutions and determine their linear stability. 
We recall that, for given parameters $\theta_1, \dots, \theta_m$, 
a steady-state solution of the system of equations \eqref{ODE00}
is a set of $n$ constants (or, more precisely, constant functions)
$u_1, \dots, u_n\in {\mathbb R}$ that satisfy
\begin{equation}
    \label{Alg}
G_i (u_1, \dots, u_n; \theta_1, \dots, \theta_m) = 0, \qquad i = 1, \dots, n. 
\end{equation}
We also recall that such a solution $U = (u_1, \dots, u_n)$
%%\BL{changed U := to U = } 
is linearly stable, if 
the linearized system of \eqref{ODE00} around $U$ is asymptotically stable, and is 
unstable otherwise. A steady-state solution $U$ is linearly stable, if 
%%Let ${\mathbf J}_G(u, \theta)$ denote the Jacobian matrix of $G$ at $u$.  
either all the eigenvalues of the corresponding Jocobian matrix at $U$
have negative real part, or all the eigenvalues of such matrix have non-positive 
real part and any eigenvalue of the matrix with zero real part has the property 
that its algebraic and geometrical multiplicities are the same; cf.\ 
\cite{Hirsch2012,Perko2001}. 

We shall consider parameters (i.e., parameter vectors) $\Theta = (\theta_1, \dots, \theta_m)$ 
in a given subset $\Omega$ of $\R^m$, and consider (steady-state) solutions (i.e., 
solution vectors) $U = (u_1, \dots, u_n)$ in a given  subset $D$ of ${\mathbb R}^n$.  
We call $\Omega$ and $D$ the {\it parameter space} and the (steady-state)
{\it solution space,} respectively. Since the steady-state solutions of \eqref{ODE00} 
and their linear stability are completely determined by the functions 
$G_i: D \times \Omega \to {\mathbb R} $ $(i = 1, \dots, n)$, 
we shall focus on the system of algebraic equations \eqref{Alg}. 
%%\BL{How about removing the following part of this paragraph? These compact notions
%%are not used, and we can cite the equations (2.1) instead of (2.2).}
This system can be expressed in the following compact form using vector notation: 
\begin{equation}
\label{eq:main}
\cG(U, \Theta) = 0,  
\end{equation}
where $
\cG(U, \Theta)=(G_1(U, \Theta), \cdots, G_n(U, \Theta))^T$, for
$U\in D$,  
$\Theta\in \Omega $, and a superscript $T$ denotes the transpose. 

%%%%%%%%%%%%%%%%%%%%%%%%%%
\begin{comment}
%The problem comes from, e.g, the study of the steady states of parameterized ODEs systems.
%Our main goal is to design and analyze neural networks for studying the solutions to \cref{eq:main}.

We assume the following: 
\begin{itemize}
    \item[(A1)] The parameter space $\Omega \subset \R^m$ is a nonempty, bounded, and open set. 
    \item[(A2)] There exist a positive integer $N$, nonempty, pairwise disjoint, and open subsets $\Omega_1, \dots, \Omega_N $ of $\Omega,$ and distinct non-negative integers $\cN_1, \dots, \cN_N$ 
    such that $\overline{\Omega} = \bigcup_{i=1}^N \overline{\Omega}_i$, and for each $i$ 
    $(1 \le k \le N)$ and each parameter $\Theta = (\theta_1, \dots, \theta_m) \in \Omega_i$, the
    system (\ref{eq:main}) has exactly $\cN_i$ solutions. 
\end{itemize}
\end{comment}
%%%%%%%%%%%%%%%%%%%%%%%%%%%

Throughout, we assume the following: 
\begin{itemize}
    \item [(A1)] 
    Both the parameter space $\Om\subset\R^m$ and 
the solution space $D\subset\R^n$ are bounded open sets.
\item[(A2)]
There exist a positive integer
$M$, pairwise disjoint open subsets $\Omega_0, \Omega_1, \dots$, $\Omega_M$ of $\Omega$, 
and distinct non-negative 
integers $\cN_0, \cN_1, \dots, \cN_M$ with $\cN_0=0$ and $\cN_i>0$ ($i=1, \cdots, M$) such that 
$\overline{\Omega} = \bigcup_{i=0}^M \overline{\Omega_i}$ and 
\begin{itemize}
\item[$\bullet$]
if $\Theta \in \Omega_0$ then 
the system \eqref{eq:main} has no solutions,
\item[$\bullet$]
if $\Theta \in \Omega_i $ with $1 \le i \le M$, then  
the system \eqref{eq:main} has exactly $\cN_i$ solutions. 
\end{itemize}
We shall call all the boundaries $\Gamma_{i, {\rm soln}} = \partial \Omega_i \cap \Omega$
$(i = 0, 1, \dots, M)$ the {\it parameter phase boundaries for solution.} 
%%\XT{In addition, we need to make an assumption on the regularity of the boundary set for the convergence/error analysis.} \BL{I suggest to postpone those assumptions to the analysis part.}
Note that we allow $\Omega_0 = \emptyset.$ 
%%\BL{I added this as it's the case for Gray--Scott.}
For each $\Theta = (\theta_1, 
\dots, \theta_m) \in \Omega_i$ with $1 \le i \le M$, we shall denote by 
\begin{equation}
    \label{eq:SThetai}
S^\Theta =\{ \hat{U}_1^\Theta,  \cdots, \hat{U}_{\cN_i}^\Theta \}  \subset D,
\end{equation}
the set of solutions to \eqref{eq:main} corresponding to $\Theta \in \Omega_i$. 
Note that each $\hat{U}_j^\Theta$ depends implicitly 
on $i$ as $\Theta \in \Omega_i.$ 
%%We shall keep this notation when no confusion arises. 
Additionally, we denote $S^\Theta =\emptyset$ if $\Theta \in \Omega_0.$

%%%%%%%%%%%%%%%%%%%%%%%%
\begin{comment}
For each $i $ $(1 \le i \le M)$, the solutions $\hat{U}_j^\Theta$ $(j = 1,\dots, \cN_i)$
depend continuously on $\Theta \in \Omega_i$, i.e., \BL{Please check this assumption - we need this for
the convergence / error analysis. Also, we may want to postpone these assumptions to the
analysis part.}
\[
\max_{1 \le k, l \le \cN_i} \left\| \hat{U}_k^{\Theta_1} - \hat{U}_l^{\Theta_2} \right\|
\to 0 \qquad \mbox{if } \Theta_1, \Theta_2 \in \Omega_i \quad { and } \quad \Theta_1 \to \Theta_2. 
\]
\end{comment}
%%%%%%%%%%%%%%%%
\item[(A3)]
For each $i \in \{ 1, \dots, M \}$, there exist disjoint
open subsets $\Omega_{i, 1}$ and $\Omega_{i, 2}$ of $\Omega_i$ such that
$\overline{\Omega_i} =  \overline{\Omega_{i, 1}} \bigcup \overline{\Omega_{i,2}}$ and 
\begin{itemize}
\item[$\bullet$]
any solution corresponding to a parameter $\Theta \in \Omega_{i, 1}$ is (linearly) unstable, 
\item[$\bullet$]
any solution corresponding to a parameter $\Theta \in \Omega_{i, 2}$ is (linearly) stable. 
\end{itemize}
We shall call all the boundaries $\Gamma_{i, {\rm stab}} = 
\partial \Omega_{i, 1} \bigcap \partial \Omega_{i, 2}$ $(i = 1, \dots, M) $
the {\it parameter phase boundaries for stability}. 
Note that we allow one of these two open sets $\Omega_{i, 1}$ and $\Omega_{i, 2}$ to be the empty set. 
%%\BL{I guess we need to assume that the stable solutions depend continuously on parameters. If so, 
%%we can make the assumptions in the analysis part.}
\end{itemize}
Additional assumptions will be made in the convergence analysis in section~\ref{sec:Analysis}.
\subsection{Target functions}
\label{subsec:TargetFunctions}

We now construct functions on the product of the solution space and the parameter space 
that can be used to 
%%\BL{removed "help"} 
identify whether specific parameters correspond to (multiple) solutions and, if so, assess the stability of those solutions. We shall call such functions {\it target functions,}
and will design and train neural networks to approximate these functions. 

We define a {\it target function for solution}, $\Phi: D \times \Om \to \R,$ by  
%%\XT{change this $N$ or the other $N$ which stands for the output vector dimension}
\begin{equation}
\label{eq:Phi}
\Phi(U, \Theta) = \sum_{i=1}^M \chi_{\Omega_i}(\Theta)\sum_{j=1}^{\cN_i} \exp\left(- \frac{| U - \hat{U}^\Theta_j|^2}{\delta(\Theta)}\right)\qquad \forall (U,\Theta)\in D\times \Om,
\end{equation}
where $\chi_{\Omega_i}$ denotes the indicator function of the set $\Omega_i$ and $\delta:  \bigcup_{i=1}^M \Om_i \to (0,\infty)$,  called a deviation function, is defined by 
\beq
\label{eq:delta}
\delta(\Theta):=
\left\{
\begin{aligned}
&
\max\left\{\frac{1}{4} \min_{\hat{U}^\Theta_j,\hat{U}^\Theta_{j'} \in S^\Theta,  j \neq j'} \|\hat{U}^\Theta_j-\hat{U}^\Theta_{j'} \|_2, \delta_0 \right\} \quad &\text{if } |S^{\Theta}|\geq 2\\
& \delta_1 \quad&\text{if } |S^{\Theta}|=1
\end{aligned}
\right.
\eeq
for a small $\delta_0>0$ to ensure $\delta(\Theta)\geq \delta_0>0$, and $\delta_1$ is taken as a portion of the diameter of the domain $D$.
%\BL{This $\delta(\Theta)$ is only defined $\Theta in \Omega_i$ with $2 \le i \le M$, right?} %\XT{Thanks. It is for $1\leq i\leq M$. The domain of $\delta$ is changed.}
Note that the function $\Phi$ is non-negative and is also considered a  
piecewise Gaussian mixture function.
%%\XT{I changed the name. Gaussian mixture is usually referred 
%%to as the weighted sum of Gaussians.} 
%%One can pick the value $\delta_i$ 
Moreover, if $\Theta \in \bigcup_{i=1}^M \Omega_i$, then the function 
$\Phi ( \cdot, \Theta): D \to \R$ is analytic. However, for each $U \in D$, 
the function $\Phi(U, \cdot): \Omega \to \R$ is only piecewise continuous or smooth, 
provided that the solutions $\hat{U}_j^\Theta$ $(j = 1, \dots, \cN_i)$ depend on 
$\Theta \in \Omega$ continuously or smoothly. 
The deviation function $\delta$ is designed to make the Gaussian ``bumps" (i.e., 
peaks of individual Gaussian radial basis functions) well separated.

We remark that the boundaries of all sets $\Omega_i$ $(0 \le i \le M)$ and the 
exact solution set $S^\Theta$ for each $\Theta\in\Om$ 
are often unknown analytically. But they can be determined numerically
by training our neural networks that approximate the target function $\Phi.$ 

%%%%%%%%%%%%%%%%%%%%%%%
\begin{comment}
\begin{remark}
    We remark that for the definition of $\Phi$ in \cref{eq:Phi}, one needs the information of the partition $\{\Om_i\}_{i=0}^N$ of $\Om$. In practice, the partition is unknown and is not given as a part of the training data in our algorithm. Moreover, our algorithm can predict the boundary of each $\Omega_i$. More detailed discussions can be found in \Cref{subsec:PSNN,subsec:solutions,sec:Results}. The values of $\del_1,\del_2,\cdots,\del_N$ are properly chosen such that the Gaussian bumps are well separated. Specific examples are discussed in numerical experiments in \Cref{sec:Results}.
\end{remark}
\end{comment}
%%%%%%%%%%%%%%%%%%%%%%%%%

We now construct a similar function for studying the stability of solutions for a given parameter. 
For any $i \in \{ 1, \dots, M\}$,  $\Theta \in \Omega_i,$ and $\hat{U}_j^\Theta \in S^\Theta$, we denote
$s_j^\Theta = 0$ if $\hat{U}_j^\Theta $ is (linearly) stable and 
$s_j^\Theta = 1$ if $\hat{U}_j^\Theta$ is (linearly) unstable. 
We define a {\it target function for stability,} $\Phi^{\rm s}: D\times \Omega \to \R,$ by 
%%\BL{Do we need some normalizing constants to make this target function within $[-1, 1]$?}
%\begin{remark}
%%In the case that the stability of solutions is also a concern, we design a signed 
%%function $\Phi^s$ that incorporates the information on the stability of solutions. 
%%More specifically, 
%%for a given $\Theta$, let $s_{j}^\Theta \in \{ 0, 1\}$ be a label that indicates stable solution for %%$s_{j}^\Theta=0$ and unstable solution for $s_{j}^\Theta=1$. Then define 
\begin{equation}
\label{eq:Phis}
\Phi^{\rm s}(U, \Theta) = \sum_{i=1}^M \chi_{\Omega_i}(\Theta) \sum_{j=1}^{\cN_i} 
(-1)^{s_{j}^\Theta} \exp\left(- \frac{| U - \hat{U}^\Theta_j|^2}{\delta(\Theta)}\right) 
\qquad \forall (U,\Theta)\in D\times \Om. 
\end{equation}
Note that if $\Theta \in \Omega_i$ for some $i \in \{ 1 , \dots, M \}$ 
and $U \in S^\Theta,$ then $U$ is stable
if and only if $\Phi^{\rm s}(U, \Theta) \approx 1$ and $U$ is unstable 
if and only if $\Phi^{\rm s}(U, \Theta) \approx -1.$
The smoothness property of $\Phi^s$ is similar to that of $\Phi$.
%Note that for a fixed $\Theta \in \bigcup_{i=1}^M \Omega_i$ the function 
%$\Phi^{\rm s}(\cdot, \Theta): 
%D \to \R$ is analytic but for a fixed $U \in D$ the function $\Phi^{\rm s}(U, \cdot): 
%\Omega \to \R$ is discontinuous and piecewise smooth provided the solutions
%in $S^\Theta$ depend on $\Theta$ smoothly. 

%In this case, a stable solution pair $(U, \Theta)$ 
%%leads to $\Phi^s(U, \Theta) \approx 1$ and and an 
%unstable solution pair $(U, \Theta)$ leads to $\Phi^s(U, \Theta) \approx -1$. 

In the following, we will mainly use the function $\Phi$ to illustrate 
our numerical algorithms to train our neural networks that approximate $\Phi$, as
these algorithms can be readily adapted for the function $\Phi^{\rm s}.$ 

%\end{remark}

%%%%%%%%%%%%%%%%%%%%%%%%%%%%%%%%%%%%%%%%%%%%%%%%%%%%%%%%%%%%%%%%
\subsection{The architecture of the parameter-solution neural network} 
\label{subsec:PSNN}

We now construct a parameter-solution neural network (PSNN) to approximate
the target function $\Phi: D \times \Omega \to \R $ that is defined in \eqref{eq:Phi}. 
Let us fix a positive integer $N.$ We first introduce a {\it parameter neural network} 
(PNN) and a 
{\it solution neural network} (SNN) 
%%\BL{I added the abbreviations PNN and SNN.}
\[
\Phi_{\text{PNN}}( \cdot, \omega_{\text{P}}):  \Om \to \R^N
\qquad \mbox{and} \qquad 
\Phi_{\text{SNN}}(\cdot,  \omega_{\text{S}}): D \to \R^N,
\]
respectively,  where $\omega_{\text{P}}$ and $\omega_{\text{S}}$ 
denote the respective sets of neural network parameters.
These are vector-valued neural networks.
%and their numbers of hidden layers and numbers of weights in each layer can be different. 
%%We use the ReLu activation function for these neural networks. 
Specifically, each of these networks is a composite function of the form
$
T_{L+1}\circ T_L\circ \cdots \circ T_1
$
with $L$ being the number of hidden layers of the network,
where all $T_j$ $(1 \le j \le L+1)$ are vector-valued functions with the form
\[
T_j(x) = a_j ((A_j x + b_j)) 
\quad \mbox{if } 
 1 \le j \le L
\quad \mbox{ and } \quad 
 T_{L+1} (x) = A_{L+1} x + b_{L+1},
 \]
with each $a_j$ an activation function, $A_j$ a matrix, and $b_j$ a vector. 
We use the ReLu activation function for our networks $\Phi_{\rm PNN}$
and $\Phi_{\rm SNN}$, i.e., each $a_j$ $(1 \le j \le L)$ is the ReLu function 
$a(x) = \max \{ 0, x\}$ for any $x \in \R$ and $a(x) \in \R^d$ with 
components $\max\{0, x_i\}$ $(i= 1, \dots, d)$ 
if $x = (x_1, \dots, x_d)\in \R^d$. 
%%For the last layer, we do not use an activation function as usual, 
The network parameters consist of all the entries of $A_j$ and $b_j$ for 
all $j = 1, \dots, L+1.$
The number of hidden layers $L$ and the weights (i.e., the entries of 
$A_j$ and $b_j$ for all $j$) for the parameter network $\Phi_{\rm PNN}$
can be different from those for the solution network $\Phi_{\rm SNN}.$
%%\XT{We may add a sentence emphasizing that the parameter and solution networks may have different $L$'s.}
%%\BL{added.}

%%\XT{Put more details of these two subnets.}\BL{I agree.}
%%\BL{I added some details, please check.}

We define $\hat{\Phi}(\cdot, \cdot, {\omega}): D\times \Omega \to \R$ by 
\[
\hat{\Phi}_{\rm PSNN} (U, \Theta, \omega)= 
\Phi_{\rm PNN} (\Theta, \omega_{\rm P}) \cdot \Phi_{\rm SNN}(\Theta, \omega_{\rm S})
\qquad \forall (U, \Theta) \in D \times \Om,
\]
where ${\omega} = \{ \omega_{\rm P},\omega_{\rm S}\}$ and 
the dot denotes the dot product of vectors in $\R^N$.
In addition, we define a scaled sigmoid function $\sigma: \R \to (-\eta, 1+\eta)$ 
\beq
\label{eq:scaledsigmoid}
\sigma(t) = \frac{e^t}{e^t + 1} + \eta \, \frac{e^t - 1}{e^t + 1}\qquad \forall t \in \R ,
\eeq
where $\eta>0$ is a small number that satisfies 
$
1+\eta > \sup_{U \in \Om, \Theta \in D}\Phi(U, \Theta)$. 
%%\BL{I combined two "sup"s into one.}
We finally define our PSNN $\Phi_{\text{PSNN}}( \cdot, \cdot, \omega): D\times \Omega
\to \R$ by 
%%\BL{Do we still have the last output layer - an affine function of the inner 
%%product, i.e., the $T$ function in Introduction?} 
%%\YZ{We're not using T now. And I have removed it.}
%and $\hat{\Phi}_{\text{PSNN}}( \cdot, \cdot, \omega)$ by 
%%$D \times \Om \to (0, 1)$ by
\begin{equation}
\label{eq:Phi_PSNN}
\Phi_{\text{PSNN}}= \sigma (\hat{\Phi}_{\text{PSNN}})
= \sigma( \Phi_{\text{PNN}} \cdot \Phi_{\text{SNN}} ).
\end{equation}
%where the symbol $\circ$ denotes the function %composition, 
%%where the dot denotes the inner product of two vectors in $\R^N$, 
%%and $\omega$ is the set of neural network parameters. 
%given by $\omega  = \omega_{\text{P}} \cup %\omega_{\text{S}} \cup \{ \alpha_{\text{PS}},
%\beta_{\text{PS}} \}.$ 
%Explicitly, we have 
%\[
%\Phi_{\text{PSNN}}(U, \Theta, \omega) 
%= S ( \alpha_{\text{PS}} ( \Phi_{\text{PNN}} (U, %\omega_{\text{P}}) \cdot
%\Phi_{\text{SNN}} (\Theta, \omega_{\text{S}}) ) %+\beta_{\text{PS}})
%\qquad \forall U \times \Theta \in D \times \Om. 
%\]
The structure of our PSNN is depicted in Figure~\ref{fig:NN}. 

Similarly, we construct the neural network $\Phi^{\rm s}_{\rm PSNN}$ 
to approximate the stability target function $\Phi^{\rm s}$. The network
$\Phi_{\rm PSNN}^{\rm s}$ is the inner product of two vector-valued
subnetworks $\Phi_{\rm PNN}^{\rm s}$ and $\Phi_{\rm SNN}^{\rm s}$ with the same 
output dimension, 
%%\BL{changed "dimensions" to "dimension."}
similar to the subnetworks $\Phi_{\rm PNN}$ and $\Phi_{\rm SNN}$, respectively. 
%However, since the stability target function $\Phi^{\rm s}$ is not scaled to be 
%in $[0, 1]$, we do not apply the sigmoid function to the inner product of the two sub networks.
Finally, $\Phi_{\rm PSNN}^{\rm s} = \Phi_{\rm PNN}^{\rm s} \cdot \Phi_{\rm SNN}^{\rm s}$.
%%
%%\BL{complete this.}

%%%%%%%%%%%%%%%%%%%%%%
%%Note that in our network architecture, the parameter-network and the solution-network 
%%output vectors with the same length, but they do not share the same structure or weights.

%%We denote by $\Phi_{\text{PSNN}}(U, \Theta; \omega)$ a PSNN approximation
%%of the target function $\Phi(U, \Theta)$, where $\om$ denotes the set of neural network parameters. 
%%\XT{Define $\Phi_{\text{PSNN}}$ as the inner product of two functions.} 

%%\BL{Need to describe some differences between the solution net and parameter net.
%%Otherwise, why call one solution and one parameter net. In particular, how the
%%regularities of the target function in $U$ and in $\Theta$ affect these nets?}
%%%%%%%%%%%%%%%%%%%%%%%%%%%%%%%%%%%%%%%%%%%%%%%%%%%%%%%%%%%%%%%%%
\subsection{Training the parameter-solution neural network}
\label{subsec:TrainingPSNN}

Assume that we are given a complete set of observation data
%\XT{Do we want to make a remark on partial data in this section? May just refer to Section 4.4.}
%%(e.g., collected from experiment in application) 
\begin{equation}
    \label{Tobserv}
\mathcal{O} = 
\left\{ O_i := \left(\Theta_i,  \{ \hat{U}_j^{\Theta_i}\}_{j=1}^{\cN_{m_i}}\right) \right\}_{i=1}^{N_{\rm observ}},
\end{equation}
where all $\Theta_1, \dots, \Theta_{N_{\rm observ}}$ 
    are distinct vectors in $\bigcup_{k=0}^M \Omega_k$, and 
    $m_i \in \{ 0, 1, \cdots, M\}$ if $i \in \{ 1, \dots, N_{\rm observ} \}$ and different labels $i$ may have the same $m_i$. 
    Moreover, 
\begin{itemize}
    \item[$\bullet$] 
    If $m_i = 0$, then $\cN_{0} = 0$ and 
    $\{ \hat{U}_j^{\Theta_i}\}_{j=1}^{0}$
    is understood as the empty set; and 
   \item[$\bullet$]
    %%If $i \in \{ 1, \dots, N_{\rm observ} \}$ and  
    %%\BL{use "If" instead of "if"?} 
    If $m_i > 0,$ then 
    $ \{ \hat{U}_1^{\Theta_i}, \dots, \hat{U}_{\cN_{m_i}}^{\Theta_i} \} =  S^{\Theta_i}$ 
    is the complete solution set corresponding to the given $\Theta;$
    cf.\ \eqref{eq:SThetai} for the notation $S^{\Theta_i}.$
    %%are solutions corresponding  to $\Theta_i$, and these solutions may not be all distinct. 
    %Moreover, $\{ \hat{U}_1^{\Theta_i}, \dots, \hat{U}_{n_i}^{\Theta_i} \} = S^{\Theta_i}.$
\end{itemize}
We shall divide the observation data set into three disjoint parts: $\cO_{\rm train}=\{ O_i\}_{i\in I_{\rm train}}$ for training,
$\cO_{\rm search}=\{ O_i\}_{i\in I_{\rm search}}$ for searching or validation, and $\cO_{\rm test}=\{ O_i\}_{i\in I_{\rm test}}$ for testing. The index sets $I_{\rm train}, $ $I_{\rm search}$, and $I_{\rm test}$ 
form a partition of the index set
$\{ 1, \dots, N_{\rm observ}\}$ of $\cO.$ 

To train our neural networks, 
%%we use part of the observation data. For simplicity of notation, 
%%we still denote that part of the data by $\mathcal{T}_{\rm observ} $ 
%%as in \eqref{Tobserv}. In addition, 
we  generate $N_{\rm random}$ points 
$\{ U_1, \dots, U_{N_{\rm random}} \}$ in the solution space $D$ and define  
%For each $\Theta_i$ $(1 \le i \le N_{\rm observ})$, we generate 
%$N_{\rm random}$ points $U_j^{i} $ $(j = 1, \dots, N_{\rm random})$ in $D$ following a uniform distribution, where $N_{\rm random}$ is a pre-chosen
%positive integer, independent of $i$. 
the training data set to be 
%\[
%\cU_i = \{ \hat{U}_1^{\Theta_i}, \dots, \hat{U}_{\cN_{m_i}}^{\Theta_i}; U_i^1, \dots, U_i^S \}.
%\]
\begin{equation}
    \label{Ttrain}
\mathcal{T_{\rm train}} = \left\{ \left( \Theta_i, \hat{U}_j^{\Theta_i}, 
\Phi(\hat{U}_j^{\Theta_i}, \Theta_i) \right)_{j=1}^{\cN_{m_i}}; 
\left( \Theta_i, U_j, \Phi(U_j, \Theta_i) \right)_{j=1}^{N_{\rm random}}
\right\}_{i \in I_{\rm train}}. 
%%{i=1}^{N_{\rm observ}}.
\end{equation}
We define the loss function
\begin{align}
\label{eq:loss}
\cL(\om) &=  \frac{1}{|\mathcal{T_{\rm train}}|}  \sum_{i \in I_{\rm train}}
%%\sum_{i=1}^{N_{\rm observ}}  
\left[ \sum_{j=1}^{\cN_{m_i}}
\left( \Phi_{\rm PSNN}(\hat{U}_j^{\Theta_i}, \Theta_i, \omega )
- \Phi (\hat{U}_j^{\Theta_i}, \Theta_i) \right)^2 
\right. 
\nonumber \\
&\qquad 
\left.  
+ \sum_{j=1}^{N_{\rm random}} \left(  \Phi_{\rm PSNN}( U_j, \Theta_i, \omega )
- \Phi (U_j, \Theta_i) \right)^2 \right], 
%frac{1}{|\cT|} \sum_{i=1}^P\sum_{j=1}^{S+\cN_{m_i}}(\Phi( U_j, \Theta_i) - \Phi_{\text{PSNN}}(U_j^i, %\Theta_i; \om))^2, 
\end{align}
%%\BL{remove this sentence in blue.}
%%{\textcolor{blue}{where $|A|$ denotes the number of elements of a set $A$, and in particular,}}
where $|\mathcal{T_{\rm train}}|= \sum_{i\in I_{\rm train}}(\cN_{m_i}+ N_{\rm random})$.
Furthermore,  the test data set $\cT_{\rm test}$ is generated in a similar fashion, utilizing $\cO_{\rm test}$ extracted from the observation set $\cO$.

\begin{remark}
%%(1)  In our implementation, we often only use part of the observation data for training. However, for simplicity of notation, here we use the same observation data set. 
%%\XT{I added some sentences here.}
%%\BL{They look good.}
In our setting, we assume that the observation data set is complete, meaning that the data set contains the entire set of solutions $\{\hat{U}_j^{\Theta_i}\}_{j=1}^{\cN_{m_i}}$ for a given parameter $\Theta_i$. 
This is, however, a restrictive assumption in practice since missing observations can occur in reality. 
    To assess the applicability of our approach in different scenarios, we also conducted tests with incomplete observations; cf.\ section~\ref{subsec: incplt data}.
\end{remark}
We train the PSNN $\Phi_{\rm PSNN}$  by minimizing the loss function
\eqref{eq:loss} using the stochastic
optimization method ADAM \cite{kingma2014adam}, and our training process
is summarized in Algorithm~\ref{alg:training}. 

%\BL{We need to add some descriptions of the implemention - e.g., which software used? what is 
%the optimization algorithm Adam, etc. Also, this algorithm is quite simple, I am not sure if 
%we need to write it down here.}

\begin{algorithm}[htpb]
\caption{The training of PSNN}
\label{alg:training}
\SetKwInput{Input}{Input}
\SetKwInput{Output}{Output}
\SetKwInput{Setup}{Set up}
\SetKwInput{Init}{Initialize}
\Input{The training data set $\mathcal{T}_{\rm train}$, the number of 
layers and number of weights in each layer of the parameter-network
$\Phi_{\rm PNN}$ and the solution-network $\Phi_{\rm SNN},$
the dimension of output vectors $N$.}
%%$\mathcal{T_{\rm train}}=\left\{(\Theta_i,U^i_j, \Phi(U_j^i, \Theta_i)), j=1,\cdots, %%S+\cN_{m_i}\right\}_{i=1}^P$}
%\Setup{Structures of the parameter-network $\Phi_{\rm PNN}$ and solution-network
%$\Phi_{\rm SNN} $}
%\Setup{the dimension of the output vector}
\Init{PSNN weights $\omega_0$} 
\For{$k=1,\dots, K_{\rm epoch}$}{
   initialize the weights of PSNN as $\omega_{k-1}$
   %%\;
   
   $\omega_k \gets \argmin\, \cL(\omega)$ by ADAM
}
\Output{PSNN weights $\omega_{K_{\rm epoch}}$}
\end{algorithm}

\vspace{-2 mm}

\begin{remark}
For learning the stability of solutions, we use the target function 
for stability $\Phi^{\rm s}$ defined in \eqref{eq:Phis}. The training data set and the loss 
function for training the PSNN for stability, $\Phi^{\rm s}_{\rm PSNN}$, 
can be constructed similarly, and the training process is also similar.
%%by replacing $\Phi$ with $\Phi^s$. 
%For the PSNN structure, we remove the sigmoid function shown in \Cref{fig:NN}. 
%We shall denote the PSNN approximation of $\Phi^s$ by $\Phi^s_{\text{PSNN}}$.
\end{remark}

%%%%%%%%%%%%%%%%%%%%%%%%%%%%%%%%%%%%%%%%%%%%%%%%%%%%%
\subsection{Locating solutions}
\label{subsec:solutions}
Utilizing the network $\Phi_{\text{PSNN}}$, 
we now present a post-processing method to
determine for any given parameter vector $\Theta \in \Omega$ if 
there exists a solution $U \in D$ corresponding to $\Theta$, and in case so, 
to locate all the solutions corresponding to  $\Theta$.

By the distinguished properties of the 
target function $\Phi: D \times \Omega \to \R$ defined in \eqref{eq:Phi},
%%cf.\ section~\ref{subsec:TargetFunctions},
once $\Theta\in \bigcup_{i=0}^M \Omega_i$ is fixed, the peaks of the graph of the function 
$U \mapsto \Phi(U, \Theta)$ defined on the solution space $D$ correspond to the multiple
solutions in $S^\Theta$, the solution set defined in \eqref{eq:SThetai}.
If there are no peaks, then there are no solutions corresponding to $\Theta.$
%%This amounts to finding the peaks of $\Phi_{\text{PSNN}}(U, \Theta; \om)$ 
%%as a function of $U$ for a given parameter $\Theta$. 
Therefore, for a given $\Theta,$ we can proceed with the following few steps to locate 
the corresponding solutions, which are also called centers since our target function 
$\Phi$ is a sum of Gaussian radial basis functions: 
%%Our method consists of the following main steps: 
%%The main steps of the our algorithm are summarized as follows: 
\begin{itemize}
%\item[{\it Step 0.}]
%Fix $\Theta.$
\item[{\it Step 1.}] 
Choose a finite set of points $\mathcal{U} \subset D$ that are uniformly scattered
in $D$ (e.g., finite-difference
grid points) and calculate the PSNN values $\Phi_{\text{PSNN}}(U, \Theta, \omega)$ 
%%(These values approximate those values of the target function $\Phi(U, \Theta)$ 
for all $U \in \mathcal{U}$.
%\item[{Step 2.}]
%Use the validation data set to select an optimal cut value $L_{cut} \in (0,1)$. 
%Determine a cut value $L_{cut} \in (0,1)$. 
\item[{\it Step 2.}] 
Choose a number $L_{\rm cut}\in (0, 1)$, call it a threshold value or a cut value.
%%\BL{How about: a threshold value or just a threshold?}
%Get the output labels $\{ y_i = \Phi_{\text{PSNN}}(U_i, \Theta; \om)  \}_{i=1}^{\cN_U}$, where $\Theta\in %\Om$ is given and  $\{U_i\}_{i=1}^{\cN_U}$ is a uniformly scattered grid on $D$.
%%\item[{Step 3.}] 
Collect all the points $U \in \mathcal{U}$ such that $\Phi_{\rm PSNN} (U, \Theta, \omega) \ge L_{\rm cut}.$
%%$U_i$ for which the output label $y_i$ is greater than $L_{cut}$. 
Denote by $\cU^\Theta_{\rm collected}$ the set of such points. 
\item[{\it Step 3.}] 
Apply the $K$-means clustering method, 
%%\BL{a ref on the K-means clustering?}
with a pre-chosen maximum number of clusters $C_{\rm max}$ and a pre-chosen
silhouette scoring number ${\rm sil} \in (0, 1),$
on the set of collected points $\cU^\Theta_{\rm collected}$
%%Implement $K$-means clustering on the collected points 
to locate the centers, i.e., approximate solutions, corresponding to the given $\Theta.$
%%to \cref{eq:main} for the given $\Theta$.
We denote by $\cU^\Theta$ the set of such centers.
\end{itemize}

It is essential to select a cut value $L_{\rm cut}$ to ensure the effectiveness of our algorithm. The ideal cut value can also vary with problems and the structures of the PSNN. 
Algorithm~\ref{alg:cutvalue} below details our method of finding such an optimal value. 
In this algorithm, we implement the PSNN and K-means clustering algorithms and  
%So, before describing our algorithm for locating solutions, we present the algorithm \ref{alg:cutvalue}
%for finding an optimal value of $L_{\rm cut}$, using the searching or validation data set
%%Recall the original observation data $\mathcal{T}_{\rm observ}$ is split to there uses, and we have
%\begin{equation}
%\label{Tsearch}
%\mathcal{T}_{\rm search} = \left\{ \Theta_i, \{ \hat{U}_j^{\Theta_i} \}_{j=1}^{\cN_{m_i}} 
%\right\}_{i \in I_{\rm search}} \subseteq \mathcal{T}_{\rm observ},
%\end{equation}
%%The idea of finding an optimal $L_{\rm cut}$ is to implement the PSNN and K-means clustering algorithm on the validation data set $\mathcal{T}_{\rm search}$, 
find an approximate optimal cut value that
minimizes the validation errors on $\cO_{\rm search}$. 
%%to obtain an optimal value of $L_{\rm cut}$
%%$ \in [\tilde{L}_{\rm min}, \tilde{L}_{\rm max}]$ 
%%by the bisection method. 
%%in Algorithm\ref{alg:cutvalue} is for all parameters in validation/test set, 
%%which is different from the average distance in \cref{tab:error_table}. 
Notice that in this algorithm,  $\mbox{dist} \left(\mathcal{U}^{\Theta_i}, S^{\Theta_i}\right)$ refers to relative distance between two sets that contains an equal number of elements. In our numerical experiments, we take 
    \begin{equation}
    \label{def:relativedist}
        \mbox{dist} \left(\mathcal{U}^{\Theta_i}, S^{\Theta_i}\right):= \min_{P\in \mathcal{P}} \frac{1}{\mathcal{N}_{m_i}}      \sum_{j=1}^{\mathcal{N}_{m_i}}
        \frac{\|U^{\Theta_i}_{P(j)}-\hat{U}_j^{\Theta_i}\|_2}{\mbox{diam}(D)},
    \end{equation}
    where $\mathcal{U}^{\Theta_i}= \{U_1^{\Theta_i}, \dots, U_{\cN_{m_i}}^{\Theta_i} \}, 
    S^{\Theta_i}=\{ \hat{U}_1^{\Theta_i}, \dots, \hat{U}_{\cN_{m_i}}^{\Theta_i} \}$
    (cf.\ \eqref{eq:SThetai}), 
    %%\BL{added "cf.\ ..."}
   $\mbox{diam}(D)$ denotes the diameter of $D$, and $\mathcal{P}$ represents the set of all perturbations on $\{1,\dots,\cN_{m_i}\}$. 
    Alternative metrics, such as the Wasserstein distance, 
    %%\BL{removed "s" from distances.} 
    can also be considered.

\begin{algorithm}[htbp]
\caption{Finding the optimal cut value in $K$-means clustering}
\label{alg:cutvalue}
\SetKwFunction{FCluster}{Cluster}
\SetKwFunction{FAverage}{Averageerror}
\SetKwProg{Fn}{Function}{}{}

\SetKwInput{Input}{Input}
\SetKwInput{Output}{Output}
\Input{A lower bound $\tilde{L}_{\min}\in [0, 1]$, an upper bound $\tilde{L}_{\max} \in [0,1]$, 
the set of neural network parameters $\om$ of 
the trained PSNN $\Phi_{\rm PSNN}$, an observation data set $\cO_{\rm search}$ for searching or validation, a set of points $\mathcal{U} \subset D$, a silhouette scoring number ${{\rm sil}_1} \in (0, 1)$.}
\Output{An optimal cut value $L_{\rm cut}$. }
\Fn{\FCluster{$\cU$}}{
\While{$2\leq j \leq C_{\max}$}{
perform $K$-means with $j$ clusters on $\mathcal{U}$ and get an average 
%%silhouette score 
score $\rm{sil}_j$ 
%%\;
}
$k \gets \underset{2\leq j\leq C_{\max}}{\arg\max}\{\rm{sil}_j\}$ 
%%\;

\eIf{$\rm{sil}_{k}\geq {\rm{sil}_1}$}{
perform $K$-means with $k$ clusters on $\mathcal{U}$ and get the set of 
centers $\mathcal{U}^\Theta$ 
%%\;
}{take the  mean on $\mathcal{U}$ and get the set of centers
$\mathcal{U}^\Theta$ }
%%\;}
\KwRet $\mathcal{U}_\Theta$ 
%%\;
}
\Fn{\FAverage{$L$}}
{
$\cE = 0$ 
%% \;
%%%\For{$(\Theta_i, \{\hat{U}^{\Theta_i}_j\}_{j=1}^{\cN_{m_i}} ) \in \mathcal{O}_v$}{

\For{$i = \min I_{\rm search}, \dots, \max I_{\rm search}$}
{
  Get $\cU^{\Theta_i}_{\rm collected}=\{ U \in \mathcal{U}: \Phi_{\text{PSNN}}(U, \Theta_i; \om)\geq L\}$ 
%%\; 

Get $\cU^{\Theta_i} = $\FCluster{ $\cU^{\Theta_i}_{\rm collected}$}
%%Get $\cU^{\Theta_i} = \FCluster{ $\cU^{\Theta_i}_{\rm collected}$}$
%% \; 
  
   %%\;
%%\eIf{$| \mathcal{U}^i_c| =\cN_{m_i}$}{
\eIf{$| \mathcal{U}^{\Theta_i}| =\cN_{m_i}$}
{
%% $\cE \gets \cE + $ distance between the two sets $\mathcal{U}^i_c$ and 
%% $ \{\hat{U}^{\Theta_i}_j\}_{j=1}^{\cN_{m_i}} $ 
%% $\cE \gets \cE + \mbox{dist} \left(\mathcal{U}^i_{\rm c}, 
%% \{\hat{U}^{\Theta_i}_j\}_{j=1}^{\cN_{m_i}} \right)$ 
 $\cE \gets \cE + \mbox{dist} \left(\mathcal{U}^{\Theta_i}, S^{\Theta_i}\right)$ 
    %%\; 
}
{$\cE \gets \cE + 1$ }
%%\;
}
%%}
%% \KwRet $\cE / |\cO_v |$ 
  \KwRet $\cE / | \mathcal{O}_{\rm search} | $ 
  %%\;
}
Take a few sample points from the interval $[\tilde{L}_{\rm min}, \tilde{L}_{\rm max}]$ and use the 
function {\FAverage} to get an approximate optimal cut value $L_{\rm cut}$.
\end{algorithm}

At last, the algorithm for locating solutions with a specified parameter vector is presented in Algorithm~\ref{alg:cluster}, which utilizes the 
trained neural network $\Phi_{\rm PSNN} $ and the $K$-means clustering with a given cut value. 
%%BL{In the alg, I changed a threshold value to an optimal threshold value. Please check it.} 
%%\XT{Looks good.}

\begin{algorithm}[htbp]
%%[htpb]
\caption{Locating solutions with PSNN and $K$-means clustering}
\label{alg:cluster}

\SetKwFunction{FCluster}{Cluster}
\SetKwProg{Fn}{Function}{}{}
\SetKwInput{Input}{Input}
\SetKwInput{Output}{Output}
\Input{A parameter vector $\Theta \in \Omega$, an optimal threshold value $L_{\rm cut}\in (0, 1)$, 
a set of points $\mathcal{U} \subset D$,
%%$= \{ U_i \}_{i=1}^{N_U} \subset D$,
the set of neural network parameters $\omega$ of 
the trained PSNN $\Phi_{\rm PSNN}$, a maximum number of clusters $C_{\rm max}$, and 
a silhouette scoring number ${{\rm sil}_1} \in (0, 1)$.}
%cut value $L_{cut}\in(0,1)$, parameter $\Theta\in \Om$, uniformly scattered points 
%%$\{U_i\}_{i=1}^{N_U}\subset D$, trained network $\Phi_{\text{PSNN}}$, 
%%maximum number of clusters $C_{\max}$, $sil_1\in (0,1)$ \;}
%%Output{the set of centers $\cU_c\subset D$ as approximate solutions \;}
\Output{A set of centers $\cU^\Theta$.}

%%Let $\cU_{\rm collected}=\{ U_i : 1 \le i \le N_U \text{ and }  \Phi_{\text{PSNN}}(U_i, \Theta; \om)\geq L_{\rm cut}\}$ 
Get $\cU^\Theta_{\rm collected}=\{ U \in \mathcal{U}: \Phi_{\text{PSNN}}(U, \Theta; \om)\geq L_{\rm cut}\}$ 
%%\; 

Get $\cU^\Theta = $ \FCluster{$\cU^\Theta_{\rm collected}$}
%%Get $\cU^\Theta = \FCluster{$\cU^\Theta_{\rm collected}$}$
%% \; 
\end{algorithm}

\medskip

%%%%%%%%%%%%%
\begin{comment}
%%\BL{Where and how do we use the test data set? }
Our method for locating the solutions relies on a good approximation of
the target function $\Phi$ by the PSNN
$\Phi_{\text{PSNN}}$, as indicated by our 
analysis of such approximations presented in  \Cref{sec:Analysis} 
and our numerical experiments of convergence given in
\Cref{subsec: Convergence Test}. 
\end{comment}
%%%%%%%%%%%%%%%%%%%

To demonstrate the key role of the neural network approximation for accurately locating 
the solutions, we shall present a naive mean-shift-based algorithm and compare it with our
PSNN-based algorithm with numerical results found in
\Cref{subsec: phasediag,subsec: incplt data}.
This mean-shift-based algorithm we use for comparison
is summarized in Algorithm \ref{alg:meanshift} in Appendix. 
%%in \Cref{appendix}. 

\begin{remark}
To determine the stability of learned solutions, we use our trained PSNN 
for stability, $\Phi^{\rm s}_{\rm PSNN}$. 
%%learned from data with stability information included. 
With input of the learned parameter-solution pair into the PSNN for stability
$\Phi^{\rm s}_{\text{PSNN}}$, we label the solution as stable if the output is closer to $1$ 
and unstable if the output is closer to $-1$. The solution stability is 
simply determined by the sign of this out put label. 
\end{remark}

%%%%%%%%%%%%%%%%Section 3%%%%%%%%%%%%%%%%%%%%%%%

\section{Approximation theory for the parameter-solution neural network}
\label{sec:Analysis}

%%\BL{Would it make sense that we first state the two main theorems and then prove them?}

%%%%We aim to develop an approximation theory for a class of functions,  including
%%%our target functions $\Phi$ for solution and $\Phi^{\rm s}$ for stability, by  
%%%our parameter-solution neural networks as defined in \Cref{subsec:PSNN}.
%% For convenience of presentation, we shall however denote $ y = U$ and $x = \Theta.$
We shall still assume that both 
$D \subset \R^n$ and $\Omega\subset \R^m$ ($n,m \geq 2$) are bounded and open subsets. Moreover, 
    $\overline{\Omega} = \bigcup_{i=0}^M \overline{\Omega_i}$, where $M \ge 1$ is an integer
    and $\Omega_0, \dots, \Omega_M$ are distinct open subsets of $\Omega$.
%%We consider 
For a target function $\Phi \in L^2(D \times \Omega),$ 
%%More assumptions will be stated later. 
%%For convenience of presentation, 
we shall denote the input variable of the function 
$\Phi$ by $(y, x)$ instead of $(U, \Theta)$. 
For any integer $k \ge 0$ and any $\sigma \in (0, 1]$, we denote by 
 $C^{0, \eta} $ with $\eta = k + \sigma$ the usual H{\"o}lder space 
$C^{k,\sigma}$.
%%%if $\eta  = k + \sigma \in (0, \infty)$. 
%%%%for some integer $k \ge 0$ and some $\sigma \in (0, 1]$. 
%%for $\eta\in (0,\infty)$ with $\eta = k +\sigma$ where $n\in \Z^+\cup\{0\}$ and $\sigma\in(0,1]$, we %%denote $C^{0,\eta}$ to be the usual H\"older space $C^{k,\sigma}$. 
%%Our main result in this section is the following universal approximation theorem.
%We shall use $C, C_1, C_2 $ et al.\ or $c, c_1, c_2, $ et al.\ to denote 
%generic positive constants. 
We shall write $C = C(\alpha, \beta, \cdots)$ or $c = c(\alpha, \beta, \cdots)$
to indicate that the quantity (e.g., a constant) $C$ or $c$ depends on other quantities
$\alpha$, $\beta,$ et al. 
%%and $\Phi = \Phi(y, x):D\times \Omega \to \R$. 

Our main result of analysis is the following universal approximation theorem: 
%\BL{more assumptions on $D$ and $\Omega$?}
%%\XT{Assumptions are added. The hyperrectangle assumption is made due to the fact that the %%approximation theory in  \cite{petersen2018optimal} is proved for functions 
%%defined on hypercubes. For $\Om$, I did not assume it is a hyperrectangle, since 
%% a piecewise $C^{0,\beta}$ function defined on $\Om$ can be trivially 
%%extended to be defined on a hyperrectangle, which also satisfies 
%%the assumptions in \cite{petersen2018optimal}.}

\begin{theorem}
\label{thm:main_approximation}
Let $\Om\subset\R^m$ be an open and bounded set and $D\subset \R^n$ be a hyperrectangle. 
Assume $\Phi \in L^2(D\times \Omega)$ satisfies that $\Phi (\cdot, x) \in C^{0, \alpha} 
(\overline{D})$ for each $x \in \Omega$ and $\Phi(y, \cdot) \in C^{0, \beta}(\overline{\Omega_i})$
%%for each $i$ 
%%$(0 \le i \le M)$ 
$(i = 0, \dots, M)$ for each $y \in D$ for some $\alpha > 0$ with 
$\lfloor \alpha \rfloor > n$ and $\beta > 0$. 
%%\BL{Assumptions on $\alpha$ and $\beta$? $\lfloor \alpha \rfloor > n$?} 
Assume also that each subset 
$\Omega_i $ $(0 \le i \le M)$ is $C^{0, \beta'}$ for some $\beta' \geq 2 \beta (m-1)/m$ with $m\geq 2$.
Then, for any $\epsilon \in (0, 1/2)$, there exists a ReLU parameter-solution neural network 
$\Phi_{{\rm PSNN}}$ defined by \eqref{eq:Phi_PSNN},  such that
\begin{equation*}
        \| \Phi - \Phi_{\text{\rm PSNN}}\|_{L^2(D\times \Omega)} < \epsilon.
\end{equation*}
More specifically, the above inequality can be attained with a parameter network 
$\Phi_{{\rm PNN}}: \Omega \to \R ^N$
with at most 
$L_1 = L_1(n, \beta)\ge 1$ layers and at most $c_1\ep^{-\frac{n(2\alpha+\lfloor\alpha\rfloor+n)}{\alpha(\lfloor\alpha\rfloor-n)}}$ nonzero weights  and a solution network $\Phi_{{\rm SNN}}: D \to \R ^N$ with at most 
$L_2 =L_2( m, \beta)\ge 1$ layers and at most $c_2 \ep^{-\frac{n(2\beta+\lfloor\alpha\rfloor+n)}{\beta(\lfloor\alpha\rfloor-n)}}$ nonzero weights for some constants $c_1 = c_1 (\alpha, n,\Phi)>0$ and 
$c_2 = c_2(\beta, m, \Phi)>0$, and an integer 
$N = O(\epsilon^{-\frac{2n}{\lfloor \alpha \rfloor-n}})$.
 In addition, all 
 the constants $L_1$, $L_2$, $c_1$, and $c_2$ may depend on the diameters of $D$ and $\Omega$. 
\end{theorem}

Observe that for the target function $\Phi$ defined in \cref{eq:Phi}, we have $\Phi(\cdot, \Theta)\in C^{0,\alpha}(\overline{D})$ for any $\Theta\in \sum_{i=0}^M \Om_i$ and any $\alpha>0$. In order for $\Phi$ to satisfy the assumptions in \Cref{thm:main_approximation}, we need to make additional assumptions regarding the regularity of solutions and the subdomains. 
%%\BL{added this sentence in case readers already forget the definition.} 
We recall that the notation $\hat{U}_j^\Theta$ is defined in \eqref{eq:SThetai}.

%%\XT{Assumptions regarding regularity are made.}
%%\BL{Good. I edited the definition a bit.}
%%\BL{How about "the system is $(\gamma, \gamma')$-regular instead of "the solution
%%is $(\gamma, \gamma')$-regular?}

\begin{definition}
      \label{def:regularity}
We say that the solutions to the system \eqref{eq:main} are $(\gamma, \gamma')$-regular 
for some $\gamma, \gamma' > 0$ if the following assumptions are satisfied: 
For each $i\in \{1, \cdots, M\}$ and $j \in \{1, \cdots, \cN_i \}$, $\hat{U}_j^{\Theta}: \Om_i \to \R^m$ is $C^{0,\gamma}(\overline{\Om_i})$; and for each $i\in \{ 0,1,\cdots, M\}$, the domain $\Om_i$ is a $C^{0,\gamma'}$ domain.
%%for some $\gamma'>0$.
\end{definition} 
\begin{lemma}
Let the solutions to \eqref{eq:main} be $(\gamma, \gamma')$-regular. 
The target function $\Phi$  defined in \eqref{eq:Phi} satisfies
the assumptions in \Cref{thm:main_approximation} on the function regularities for any $\alpha > 0$ and $\beta =\min (1,\gamma)$. In addition, the subdomains $\{ \Om_i\}_{i=0}^M$ satisfy the assumptions in \Cref{thm:main_approximation} if $\gamma'\geq 2\beta(m-1)/m$. 
\end{lemma}
\begin{proof}
 It is obvious that $\Phi$ is $C^{0,\alpha}(\overline{D})$ regular in 
 %%their 
 its first variables for any $\alpha>0$. To get $\beta = \min(1, \gamma)$, notice that $\delta(\Theta)$ defined by \eqref{eq:delta} is at most Lipschitz continuous.    
\end{proof}

The proof of \Cref{thm:main_approximation} consists of two major steps. In the first step, we approximate $\Phi(y, x)$  by a truncated series under an orthonormal basis on $L^2(D)$, and the variables $y$ and $x$ are separated in this step. 
In the second step, we use universal approximation theories for functions of $y$ and $x$ separately and conclude the approximation estimates for the original function $\Phi(y, x)$.
The main tool of the first step of the proof is the classical Mercer's expansion for symmetric positive definite kernels \cite{10.1007/11776420_14}. 

We proceed by defining and examining some properties of
Mercer's kernel $K_\Phi$ and the corresponding operator 
$T_{\Phi} = T_{K_{\Phi}}$ associated with $\Phi$. 

%\noindent\textbf{Approximation theory} -- Step 1.  $\Phi(y,x) \approx \bm{\Psi}(y) \cdot \bm{\Xi}(x)$

\begin{definition}\label{def:kernel}
%%Assume $\Phi \in L^2(D \times \Omega).$  
Given open and bounded sets $D \subset \R^n$ and $\Om\subset \R^m.$
Mercer's kernel associated with a function $\Phi \in L^2(D \times \Omega)$ 
is the function $K_\Phi: D \times D \to \R$ defined by 
\begin{equation}
\label{KPhi}
 K_\Phi(y,z)= \int _{\Omega} \Phi(y, x )\Phi(z, x) dx \qquad \forall y, z \in D.
\end{equation}
The operator $T_\Phi: L^2 (D) \to L^2(D)$ associated with Mercer's kernel $K_\Phi$ is defined as follows: 
for any $\phi \in L^2(D),$
\[
(T_\Phi \phi) (y) = \int _{\Omega} K_\Phi(y,z)\phi(z)\, dz \qquad \forall y \in D. 
\]
\end{definition}

We denote by $C(\overline{D}; L^2(\Om))$ the class of measurable functions 
$\Phi: \overline{D}\times \Omega \to R$
such that $\Phi( y, \cdot) \in L^2(\Omega)$ for any $y \in \overline{D}$ and 
%%maps $\Phi: \overline{D} \to L^2(\Om)$, with $\Phi(y, \cdot) \in L^2(\Om)$ for 
%%each $y \in \overline{D}$, such that 
\begin{equation}
\label{yyp}
\| \Phi(y, \cdot) - \Phi(z, \cdot) \|_{L^2(\Omega)} \to 0
\quad \mbox{if }  y, z \in \overline{D}  \mbox{ and }  | y - z | \to 0.
\end{equation}
%%One verifies that $C(\overline{D}; L^2(\Om))$ 
This is a Banach space with the norm
%%\BL{Check this.}
$\| \Phi \| = \max_{y\in \overline{D}} \| \Phi(y, \cdot) \|_{L^2(\Omega)}$.
%\[
%\| \Phi \| := \max_{y\in \overline{D}} \left(\int_{\Om}  [\Phi(y, x)]^2 dx \right)^{1/2}  
%\qquad \forall \Phi\in C(\overline{D}; L^2(\Om)). 
%\]
Clearly, we have that $C(\overline{D}, L^2(\Omega))\subset L^2(D\times \Omega)$ and 
$ \| \Phi \| \le \| \Phi \|_{L^2(D\times \Omega)}/\sqrt{| D |}$
for any $\Phi \in C(\overline{D}, L^2(\Omega)).$
%%\BL{Measurable?}

\begin{proposition}
\label{prop:Mercer}
Given $\Phi \in L^2(D \times \Omega).$ 

{\rm (1)} Mercer's kernel $K_\Phi: D\times D\to \R$
is well defined, and $K_\Phi\in L^2(D\times D).$ Moreover, it has the following properties: 
\begin{itemize}
    \item[$\bullet$] Symmetry:  $K_\Phi(y,z) = K_\Phi(z,y)$ for all $y, z \in D;$
\item[$\bullet$] 
Positive semi-definiteness: 
$\sum _{i=1}^{r}\sum _{j=1}^{r} K_\Phi(y_i, y_j)c_i c_j \ge 0$ 
for  any $r \in \N$, $y_1,\dots, y_r \in D$, and  $c_1, \dots, c_r\in \R$; and 
\item[$\bullet$] Continuity: 
If in addition $\Phi \in C(\overline{D}; L^2(\Omega))$ then 
%$(\cdot, x ) \in C(\overline{D})$ for each $x \in \Omega$ and 
%the continuity is uniform with respect to $x \in \Omega$, then
$K_\Phi \in C(\overline{D\times D}).$
%\BL{Without the uniform continuity, how to show that $K_\Phi$ is continuous?}
%\XT{The continuity of $K_\Phi$ depends on the continuity of $\Phi$ in its first variable.}
%\BL{The proof given below needs the uniform continuity.}
\end{itemize}

{\rm (2)} The operator $T_\Phi: L^2 (D) \to L^2(D)$ is well defined. Moreover, it is 
linear, self-adjoint, positive semi-definite, and compact, and hence, with countably many 
nonnegative eigenvalues. 
\end{proposition}

\begin{proof}
(1)    By Fubini's Theorem and the assumption that $\Phi \in L^2(D\times \Omega)$, we have
$\Phi(y, \cdot) \in L^2(\Omega)$ for a.e.\ $ y \in D.$ Thus, 
$\Phi(y, \cdot) \Phi(z, \cdot) \in L^1(\Omega)$ for a.e.\ $(y, z) \in D\times D.$
Consequently, the integral in \eqref{KPhi} is well defined for a.e.\ 
$(y, z) \in D\times D.$ By defining this integral value to be $0$ for $(y, z)$
in a subset of $D \times D$ of measure $0$, we see that 
 Mercer's kernel $K_\Phi: D \times D \to \R$ is therefore well defined. 
 Applying twice the Cauchy--Schwarz inequality and using the fact that $D$ is bounded, 
 we obtain by direct calculations that $K_\Phi \in L^2(D \times D).$
The symmetry follows from the definition of the kernel $K_\Phi.$
The positive semi-definiteness follows from direct calculations: 
\[
\sum _{i=1}^{r}\sum _{j=1}^{r} K_\Phi(y_i, y_j)c_i c_j
=\int _{\Omega} \bigg(\sum _{i=1}^{r}c_i\Phi(y_i, x )\bigg)^2 dx \ge 0.
\]
%%Notice that by the uniform continuity each $y \in \overline{D}$ corresponds to 
%%some $\delta_y > 0$ such that  $\Phi(y', x)$ is 
%%bounded for all $y' \in B(y, \delta_y) \cap \overline{D}$
%%and all $x \in \Omega.$ Since $\overline{D}$ is compact, it can be covered by
%%finitely many such balls $B(y, \delta_y)$. Hence  $\Phi$ is bounded 
%%on $\overline{D} \times \Omega.$ 
Suppose $\Phi \in C(\overline{D}; L^2(\Omega)).$ 
Let $(y, z), (y', z') \in \overline{D\times D}.$ We have by \eqref{yyp} and 
the Cauchy--Schwarz inequality that 
%%\BL{Please check this.}
%%We have by the uniform continuity and the boundedness  of $\Phi$ that 
\begin{align*}
&K_\Phi (y, z) - K_\Phi (y', z') 
\\
& \quad = \int_\Omega \Phi(y, x) \Phi(z, x) \, dx 
- \int_\Omega \Phi(y', x) \Phi(z', x)\, dx
\\ 
& \quad = \int_\Omega (\Phi(y, x) - \Phi(y', x) ) \Phi(z, x) \, dx  
+ \int_\Omega \Phi(y', x) (\Phi(z, x) - \Phi(z', x))\, dx  \\
& \quad \le \| \Phi(y', \cdot) - \Phi(y, \cdot) \|_{L^2(\Omega)} \, 
\| \Phi (z, \cdot)  \|_{L^2(\Omega)}
+ \| \Phi (z, \cdot) - \Phi (z', \cdot) \|_{L^2(\Omega)}
\| \Phi (y', \cdot)  \|_{L^2(\Omega)}
\\
& \quad 
\to 0 \qquad \mbox{as } (y', z') \to (y, z),
\end{align*} 
proving the continuity.

(2) This is a standard result; cf.\ \cite{Conway} (Proposition 4.7, Chapter II).
\end{proof}

\begin{remark}
 The symmetry, positive semi-definiteness, and continuity
of Mercer's kernel $K_\Phi$
correspond to the assumptions of a kernel function in Mercer's Theorem
\cite{mercer1909,10.1007/11776420_14}. This justifies the use of the term ``Mercer's kernel". 
\end{remark}

The next lemma uses this fact to represent $\Phi$ by a series expansion where the variables $y$ and $x$ are separated.
We remark that \cite{schwab2006karhunen} discussed a similar expansion under an abstract framework using Hilbert space theory. Here we present a more straightforward proof with a stronger assumption that  $\Phi$ is continuous in its first variable. 

%\BL{What is the definition of $C(\overline{D}; L^2(\Om))?$} 
%\XT{Can you check the definition added in the above?}
%%\BL{I suggest to use $\N$ instead of $\Z^+$.} \XT{Changed.}

\begin{lemma}[Kernel decomposition]
\label{lem:Mercer_expansion}
 Let $\Phi\in C(\overline{D}; L^2(\Om))$. 
 %%\BL{Same question as above: Is this $\Phi$ measurable and hence in $L^2(D\times \Om)$?} 
 Then for each $N\in\N$, there exist $\bm{\Psi} \in C(\overline{D}; \R^N)$ and $\bm{\Xi}\in L^2(\Omega; \R^N)$ such that 
 \[
  \|\Phi - \bm{\Psi} \cdot \bm{\Xi} \|_{L^2(D\times \Omega)}^2 = \sum_{k=N+1}^{\infty} \lambda_k,
  \]
  where $(\bm{\Psi} \cdot \bm{\Xi}) (y, x) : = \bm{\Psi}(y) \cdot \bm{\Xi}(x)$ and $\{ \lambda_k \}_{k=1}^\infty$ is the sequence of all the (non-negative) eigenvalues of the operator 
  $T_{\Phi}: L^2(D) \to L^2(D).$
\end{lemma}
\begin{proof}
By Proposition~\ref{prop:Mercer} and Mercer's Theorem
%%  Since $K_\Phi$ is a continuous symmetric positive semi-definite kernel, by Mercer's theorem 
\cite{mercer1909,10.1007/11776420_14}, 
there exists a complete orthonormal basis $\{ e_j \}_{j=1}^\infty$ of $L^2(D)$ 
consisting of eigen-functions corresponding
to the sequence of all the (nonnegative) eigenvalues $\{ \lambda_j \}_{j=1}^\infty$ of 
the operator $T_{\Phi}: L^2(D) \to L^2(D)$ such that 
%%there exist a sequence of non-negative eigenvalues $\{\lambda_j\}_{j=1}^\infty$ 
%%and the corresponding sequence of eigenfunctions $\{e_j\}_{j=1}^\infty $ of the 
%%operator $T_{\Phi}: L^2(D) \to L^2(D)$ such that 
each $e_j\in C(\overline{D})$ $(j \ge 1)$
%%$\{ e_j \}_{j=1}^\infty$ forms a complete orthonormal basis of $L^2(D)$, 
and $K_\Phi(y,z)$ can be expressed as
\begin{align}
\label{KPhi_expansion}
    K_\Phi(y, z)=\sum_{j=1}^{\infty} \lambda_j e_j(y) e_j(z) \qquad \forall y, z \in \overline{D}, 
\end{align}
where the infinite series 
%%Note that each $e_j \in C(\overline{D})$ and the above series 
converges absolutely and uniformly on $\overline{D}$. 
%%\BL{removed this: ``Let $y \in \overline{D}.$  Since $\Phi(y, \cdot)\in L^2(\Omega)$"}
Since for a.e.\ $x \in \Omega$,  $\Phi(\cdot,x) \in L^2(D)$, we have 
the following $L^2(D)$-expansion:
%%of the function $\Phi(\cdot, x) \in L^2(D):$
%%for each $y \in \overline{D}$ that 
%%Fixing $x\in \Om$, we can expand $\Phi(y, x)$ as
\begin{align*}
    \Phi(y,x)=\sum_{j=1}^{\infty} e_j(y) \varphi_j(x)\qquad \mbox{a.e. } y \in D, 
\end{align*}
%%\BL{changed $x \in \Omega$ to $y \in D$}
with the series converging in $L^2(D),$
where 
\[
\varphi_j(x) = \int_{D} \Phi(y, x) e_j(y) \, dy \qquad \forall j \ge 1.
\]
%%\BL{I removed "and the series converges in $L^2(\Omega).$" 
%%- that's something proved below - see (3.6).}
Clearly $\phi_j \in L^2(\Omega)$ for each $j \ge 1$. Moreover, we have for any $i, j \ge 1$ that
%Notice that the following property holds for $\{\varphi_j \}_j$, 
%%for any $i, j \ge 1$ we have 
\begin{equation}
\label{phiOrtho}
%\begin{split}
    \int_{\Omega} \varphi_i(x) \varphi_j(x) \,dx 
    %%&= \int_{\Omega} \left( \int_D \Phi(y, x)e_i(y) \, dy\int_D \Phi(z, x)e_j(z) \, dz\right) dx\\
    %%&=\int_{D} \left( \int_D e_i(y)e_j(z) \int_{\Omega} \Phi(y, x) \Phi(z, x) \, dx dy\right) dz\\
    %%&
    =\int_{D} e_i(y) \bigg( \int_D K_\Phi(y, z)e_j(z) dz\bigg) \, dy 
    %\\
   % &
   %%=\int_{D} e_i(y) \lambda_j e_j(y) \, dy
    =\lambda_j \delta_{ij}.
%\end{split}
\end{equation}

Let $\bm{\Psi} = (e_1,\dots, e_N)$ and $\bm{\Xi}= (\varphi_1,\dots, \varphi_N)$, and denote
%%\begin{equation}
%%    \label{PsiXi}
$(\bm{\Psi} \cdot \bm{\Xi})(y, x):= \bm{\Psi}(y) \cdot \bm{\Xi}(x) = \sum_{j=1}^{N} e_j(y) \varphi_j(x).$
%%\end{equation}
Then for each $y\in D$,  one can show that 
\begin{equation}
\label{eq:truncation_error}
\begin{split}
    &\int_{\Omega} \bigg( \Phi(y, x)-\bm{\Psi}(y) \cdot \bm{\Xi}(x)\bigg)^2 \, dx\\
  =&\int_{\Omega} \Phi(y,x)^2 dx + \int_{\Omega} (\bm{\Psi}(y) \cdot \bm{\Xi}(x))^2dx-2 \int_{\Omega} \Phi(y,x) \bm{\Psi}(y) \cdot \bm{\Xi}(x) \, dx\\
    =&K_\Phi(y,y) - \sum_{k=1}^{N} \lambda_k e_k(y)^2  =  \sum_{k=N+1}^{\infty} \lambda_k e_k(y)^2 .
\end{split}
\end{equation}
Indeed, it follows from 
%%\eqref{PsiXi}, 
\eqref{phiOrtho} and \eqref{KPhi} that 
\begin{align*}
     &\int_{\Omega} (\bm{\Psi}(y) \cdot \bm{\Xi}(x))^2dx 
    %% =\int_{\Omega} \sum_{k=1}^{N}\sum_{l=1}^{N} \varphi_k(x)\varphi_l(x) e_k(y) e_l(y) \, dx\\
      %%=&
      =\sum_{k=1}^{N}\sum_{l=1}^{N}e_k(y) e_l(y) \int_{\Omega} \varphi_k(x)\varphi_l(x) \, dx
      =\sum_{k=1}^{N} \lambda_k e_k(y)^2, 
\end{align*}
and 
\begin{align*}
\int_{\Omega} \Phi(y,x) \bm{\Psi}(y) \cdot \bm{\Xi}(x) \, dx 
= &\sum_{k=1}^{N} e_k(y) \int_{\Omega} \Phi(y,x) \varphi_k(x) \, dx\\
    =&\sum_{k=1}^{N} e_k(y)\int_{\Omega} \Phi(y,x)\left( \int_D \Phi(z,x) e_k(z)\, dz \right) dx\\
     =&\sum_{k=1}^{N} e_k(y) \int_D e_k(z) K_\Phi(y,z) \, dz =\sum_{k=1}^{N} \lambda_k e_k(y)^2,
\end{align*}
which together imply \eqref{eq:truncation_error}.
It then follows from the uniform convergence of \eqref{KPhi_expansion} that
%%%%%%%%%%%%%%%
\begin{comment}
\BL{It seems that the uniform convergence is the only result from Mercer's theorem. If 
this is the case then we don't need to use Mercer's theorem, in particular, we 
only need $\Phi$ to be in $L^2(D\times \Omega).$ The convergence below (exchange of integral
and infinite sum) can be proved by Monotone Conv Thm.}
\XT{We quoted Mercer's theorem to give an expansion of the kernel in \eqref{KPhi}, which is then used later on, e.g., in \eqref{eq:truncation_error}. If the kernel is not continuous, then \eqref{KPhi} may be understood in the $L^2(D\times D)$ sense. However, in this case, we cannot understand $K_\Phi (y,y) =\sum_{k=1}^\infty \lambda_k e_k(y)^2 $, even in the $L^2(D)$ sense. There are alternative arguments, using only the Hilbert space theory, as we commented at the beginning of this lemma. I think our version is an easy corollary of Mercer's theorem.}
\end{comment}
%%%%%%%%%%%%%%%%%%
\begin{align*}
  \|\Phi - \bm{\Psi} \cdot \bm{\Xi} \|_{L^2(D\times \Omega)}^2  &=\int_D \sum_{k=N+1}^{\infty} \lambda_k e_k(y)^2  dy= \sum_{k=N+1}^{\infty} \lambda_k. 
\end{align*}   
The proof is complete.
\end{proof}

%\begin{lemma}[Mercer's theorem]\label{lemma:mercer}
%    Suppose $K$ is a continuous symmetric positive semi-definite kernel on $D\times D$. Then there is an orthonormal basis $\{e_i\}_i$ of $L^2(D)$ consisting of eigenfunctions of $T_k$ such that the corresponding sequence of eigenvalues $\{\lambda_i \}_i$ is nonnegative. The eigenfunctions corresponding to non-zero eigenvalues are continuous on $D$, and $K$ has the representation
%        $K(s,t)=\sum_{j=1}^{\infty} \lambda_j e_j(s) e_j(t)$
%    , where the convergence is absolute and uniform.
%\end{lemma}

By \Cref{lem:Mercer_expansion},  the decay rate of the eigenvalues $\{ \lambda_k\}_{k=1}^\infty$ determines the accuracy of approximating $\Phi$ by a dot product of two $N$-dimensional vector-valued 
functions. In general, the decay rate of $\{ \lambda_k\}_{k=1}^\infty$ depends on the smoothness of $K_\Phi$. 
We now quote
\cite[Proposition 2.21] {schwab2006karhunen} that assumes the Sobolev regularity of 
the kernel function. 
%We recall that the integral operator $T_K: L^2(D) \to L^2(D)$ associated to
%a given function $K \in L^2(D\times D)$ is given by \eqref{KPhi} with 
%$T_K$ and $K$ replacing $T_\Phi$ and $\Phi$, respectively. It is linear
%and compact. If $K$ is symmetric (i.e., $K(y, z) = K(z, y)$ for 
%a.e.\ $(y, z) \in D\times D$) and semi-positive definite (i.e., $T (\phi) \ge 0$ a.e.\ in $D$ if $\phi \in L^2(D)$ satisfies $\phi \ge 0$ a.e.\
%in $D$),  then all the eigenvalues of $T$ are nonnegative and there exist a sequence of the corresponding eigenvecotors that form a complete orthonormal basis for $L^2(D); $ cf.\ \cite{Conway}. 
%%In the following, we use 
We denote by $H^p(D)$ for integer $p \ge 1$ 
the standard Sobolev space of functions with weak derivatives up to order $p$ in $L^2(D)$.

%%\BL{Should we put a definition of the space $H^p(D)?$} \XT{Done.}

\begin{lemma}[Proposition 2.21 in \cite{schwab2006karhunen}]\label{lemma:eigenvalue}
    Let $K \in L^2(D \times D)$ be a symmetric and positive semi-definite kernel, $T_K$  its associated linear integral operator, and $\{\lambda_k \}_{k=1}^\infty$ the 
    sequence of eigenvalues of $T_K$.  
        If $K \in H^p(D)\otimes H^p(D)$ for some integer $p \ge 1,$
        %%% with $p > 0$,  
        then there exists a constant $C>0$ depending only on $K$ such that
        $
        0 \leq \lambda_k \leq C  k^{-p/n} $ for all $k \ge 1$.
   %%     \quad \forall k>0.
        Moreover, the eigenfunctions of $T_K$ are in $H^p(D)$.

\end{lemma}

%\begin{assu}
%\label{assu:Phi}
%We assume that $\Phi(y, x)$ is analytic in $y\in D$, and piecewise smooth in $x\in \Om$. 
%\XT{Change Assumption to Lemma, which establishes the property of $\Phi$}
%\end{assu}

\begin{theorem}
\label{thm:truncatedseries}
Let $\Phi \in H^p(D)\otimes L^2(\Om)$ with $p>n$. For each $N\in {\mathbb N}$,
   %% \BL{changed $Z_+ $ to $\mathbb N$; same for below.} 
    there exist $\bm{\Psi} \in H^p(D; \R^N)$ and $\bm{\Xi}\in L^2(\Omega;\R^N)$ such that 
    \[
\| \Phi - \bm{\Psi} \cdot \bm{\Xi} \|_{L^2(D\times \Omega)} \leq C_1 N^{-\frac{p-n}{2n}},
\]
where $C_1> 0$ is a constant  depending only on $p$, $n$ and $\Phi$. 
   
\end{theorem}
\begin{proof}
We first show that $\Phi\in H^p(D)\otimes L^2(\Om)$ implies $K_\Phi \in H^p(D)\otimes H^p(D)$. Since $K_\Phi$ is symmetric, we only need to show $K_\Phi \in H^p(D)\otimes L^2(D)$.  Indeed, for any multi-index $\al$ with $0\leq |\al|\leq p$, 
we have $D^\al_y \Phi \in L^2 (D) \otimes L^2(\Omega)$, and for any test function $g\in C^\infty_c(D)$, 
    \begin{align*}
    \int_{D}\int_{\Omega} D^\al_y  \Phi(y,x) \Phi(z,x) dx g(y)dy &= \int_{\Omega}\int_{D}D^\al_y \Phi(y,x) g(y)dy\Phi(z,x) dx\\
    &=(-1)^{|\al|}\int_{\Omega}\int_{D} \Phi(y,x)D^\al g(y) dy\Phi(z,x) dx\\
    &=(-1)^{|\al|}\int_{D}\int_{\Omega}\Phi(y,x)\Phi(z,x)dx D^\al g(y) dy\\
    &=(-1)^{|\al|}\int_{D}K_{\Phi}(y,z)D^\al g(y) dy.
    \end{align*}
   Therefore, by the definition of weak derivative, $D^\al_y K_{\Phi}(y,z) = \int_{\Omega} D^\al_y  \Phi(y,x) \Phi(z,x) dx$, and  $D^\al_y K_{\Phi}\in L^2(D)\otimes L^2(D)$ since
    \begin{align*}
    \int_{D\times D} \left|\int_{\Omega} D^\al_y \Phi(y,x) \Phi(z,x)dx \right|^2dydz&\leq 
    \int_{D\times D} \int_{\Omega}D^\al_y \Phi(y,x)^2 dx \int_{\Omega}\Phi(z,x)^2 dxdydz\\
    &= \| D^\al_y \Phi \|_{L^2(D) \otimes L^2(\Omega)} \| \Phi\|_{L^2(D) \otimes L^2(\Omega)} <\infty. 
    \end{align*}
     By \Cref{lemma:eigenvalue}, then there exists a constant $C>0$ depending only on $K_{\Phi   }$ such that $ 0 \leq \lambda_k \leq C  k^{-p/n}$ 
   %%  \BL{changed $p/d$ to $p/n.$} 
     for all the eigenvalues $\{\lambda_k\}$ of $T_{\Phi}$, and the eigenfunctions $\{e_k\}$ are $H^p(D)$.
     Notice also that 
 by applying Sobolev embedding theorem with $p>n$ (see, e.g., \cite{adams2003sobolev}), $\Phi$ is continuous in its first variable on $\overline{D}$.
 It then follows from \Cref{lem:Mercer_expansion} that by taking $\bm{\Psi} = (e_i)_{i=1}^N \in H^p(D;\R^N)$ and $\bm{\Theta} = (\int_{D} \Phi(y,\cdot)e_i(y)dy)_{i=1}^N\in L^2(\Om;\R^N)$, we have
 \begin{align*}
         \|\Phi - \bm{\Psi} \cdot \bm{\Xi} \|_{L^2(D\times \Omega)}^2 = \sum_{k=N+1}^\infty \lambda_k  & \leq \sum_{k=N+1}^{\infty} C k^{-p/n}\leq  C \int_{N}^{\infty} t^{-p/n} dt\leq \tilde{C} N^{-\frac{p-n}{n}}.
    \end{align*}
   
\end{proof}

%\noindent\textbf{Approximation theory} -- Step 2.  $\bm{\Psi}\approx \bm{\Psi}_{\text{NN}} $, $\bm{\Xi}\approx \bm{\Xi}_{\text{NN}}$\\

In the next step, we use two neural networks to approximate the functions $\bm{\Psi}$ and $\bm{\Xi}$ in \Cref{thm:truncatedseries}.
We now quote two approximation results for H\"older continuous functions and piecewise H\"older continuous functions from \cite{petersen2018optimal}. Note that in the original paper, functions are defined on a unit hypercube in $\R^d$. 
It is straightforward to adapt the arguments so that they can be applied to any hyperrectangles in  $\R^d$. 
%%\XT{Assumption on $\sD$ is added based on \cite{petersen2018optimal}}
%\begin{lemma}[Proposition 2.23 in \cite{schwab2006karhunen}]\label{lemma:eigenfunction}
 %   Assume the kernel defined by integral (as \Cref{def:kernel})is piecewise analytic in $H^{p,q}$ on $D \times D$. Then the eigenfunctions $\{e_m \}_m$ are analytic in $H^{p}$ in every $D_j$.
%\end{lemma}

\begin{lemma}[Theorem 3.1 and Corollary 3.7 in \cite{petersen2018optimal}]\label{lem:NNApproximation}
Let $\sD\subset\R^d$ be an open and bounded set and $d\geq 2$. Let $f: \sD \to \R$ and $\epsilon \in (0, 1/2)$. 
Assume either one of the following conditions holds: 
\begin{enumerate}
 \item[{\rm (1)}] $\sD$ is a hyperrectangle and $f\in C^{0,\beta}(\sD)$; 
 \item[{\rm (2)}] $f$ is piecewise $C^{0,\beta}$ with respect to a finite partition $\{\sD_j\}_j$ of $\sD$, where each $\sD_j$ is a $C^{0,\beta'}$ domain for $\beta'\geq 2\beta(d-1)/d $. 
    \end{enumerate}
%\begin{enumerate}
% \item $f\in C^{k,\alpha}(\sD)$ is $C^{k,\alpha}(\sD)$, there exists $\Phi^f_\epsilon$ with at most $(2+\lceil \log_2 (k+\alpha)\rceil)(11+(k+\alpha)/d)$ layers, and at most $c\epsilon^{-d/(k+\alpha)}$ nonzero weights where $c(d,k,\alpha, B)$ is a constant depending on $d,k,\alpha, B$.
% \item Let $\{\sD_j\}_j$ be a finite partition of $\sD$. Let $\beta'=d(k+\alpha)/(2(d-1))$ and $\beta=\mathrm{max}\{k+\alpha,  \beta'\}$. If for each $\sD_j$, $\partial \sD_j$ is $C^{k, \alpha}$, $f$ is $C^{k',\alpha'}(\sD_j)$, and $k'+\alpha' \geq \beta'$, there exists $\Phi^f_\epsilon$ with at most $(4+\lceil \log_2\beta\rceil)(12+(3\beta)/d)$ layers, and at most $c\epsilon^{-2(d-1)/(k+\alpha)}$ nonzero weights where $c$ is a constant depending on $d,k,\alpha, B$ and $\{  \sD_j\}_j$.
%\end{enumerate}
Then there is a neural network $\Phi^f_\epsilon$ with at most $L$ layers, and at most $c\epsilon^{-d/\beta}$ nonzero weights, where $L\in \N$ depends only on $d$ and $\beta$ and $c >0$ depends on $d ,\beta, f$, and also, in the second case, the partition $\{ \sD_j\}_j$ such that
$
  \|\Phi^f_\epsilon-f\|_{L^2(\sD)} \leq \epsilon
$.
\end{lemma}
%%\XT{The following sentence is added to explain (2).} 
Note that for a function $f$ that satisfies the assumptions in \Cref{lem:NNApproximation} {(2)}, one can perform a trivial extension (by zero) of $f$ to a hyperrectangle that contains $\sD$. Such an extended function satisfies the assumptions in \cite[Corrollary 3.7]{petersen2018optimal}.

%In the following we present the proof of  \Cref{lemma:approximation_thm}.
%\noindent {\it Proof of \Cref{lemma:approximation_thm}}.

We are now ready to prove \Cref{thm:main_approximation}.

\begin{proof}[Proof of {\Cref{thm:main_approximation}}]
By the assumption on $\Phi$, we have $\Phi\in H^{\lfloor \alpha\rfloor}(D)\otimes L^2(\Om)$.
By the definition of the scaled sigmoid function $\sig$ in \cref{eq:scaledsigmoid}, we have $\sig^{-1}\circ \Phi \in H^{\lfloor \alpha\rfloor}(D)\otimes L^2(\Om)$. 
Next, we prove that $\hat{\Phi}_{\rm PSNN}$ approximates $\sig^{-1}\circ \Phi$ with the desired rates. 
We first apply \Cref{thm:truncatedseries} to $\sig^{-1}\circ\Phi$.
For a given $\ep>0$,  let $N\in\N$ be large enough such that $C_1 N^{-\frac{\lfloor\alpha\rfloor - n}{2n}}< \ep /2$, where $C_1$ is the constant in \Cref{thm:truncatedseries}(1) where we replace $\Phi$ with $\sig^{-1}\circ\Phi$ for its definition. 
Then there exists $ \bm{\Psi}\in H^{\lfloor \alpha \rfloor}(D; \R^N)$ and $\bm{\Xi}\in L^2(\Om;\R^N)$ such that $ \| \sig^{-1}\circ\Phi - \bm{\Psi} \cdot \bm{\Xi}\|_{L^2(D\times \Omega)} < \ep/2 $.
In addition, since $\sig^{-1}\circ\Phi$ is $C^{0,\alpha}$ in its first variable and piecewise $C^{0,\beta}$ in its second variable, one can argue that $\bm{\Psi}\in C^{0,\alpha}(D;\R^N)$ and $\bm{\Xi}$ is piecewise $C^{0,\beta}$ with respect to the partition $\{ \Om_i\}$. 
Let $\bm{\Psi}_{\text{NN}}:D \to \R^N$ and $\bm{\Xi}_{\text{NN}}:\Omega \to \R^N$ be two ReLu network approximation to $\bm{\Psi}$ and $\bm{\Xi}$ respectively such that, for every $1 \leq j \leq N$,
\begin{equation*}
\|e_j -  (\bm{\Psi}_{\text{NN}})_j \|_{L^2(D;\R^N)}\leq \frac{\ep}{\sqrt{N}P} \quad \text{and}\quad  \|\varphi_j -  (\bm{\Xi}_{\text{NN}})_j \|_{L^2(\Omega;\R^N)}\leq \frac{\ep}{\sqrt{N}P},
\end{equation*}
which imply
%%\BL{$M$ is different from that $M$ in $\Omega_0, \dots, \Omega_M$? 
%%Use a different notation? }
%%\XT{Changed $M$ to $P$.}
\beq
\label{eq:NNapproximation}
\|\bm{\Psi} -  \bm{\Psi}_{\text{NN}} \|_{L^2(D;\R^N)}\leq \frac{\ep}{P} \quad \text{and}\quad  \|\bm{\Xi} -  \bm{\Xi}_{\text{NN}} \|_{L^2(\Omega;\R^N)}\leq \frac{\ep}{P},
\eeq
where $P>0$ satisfies $ \sqrt{\sum_{j=1}^N \lambda_j}/P+ \sqrt{N}/P + 1/P^2 < 1/2$. 
Then by the triangle inequality and Cauchy--Schwartz inequality, one can show that 
\begin{equation*}
\begin{split}
\| &  \bm{\Psi} \cdot \bm{\Xi}  - \hat{\Phi}_{\rm PSNN} \|_{L^2(D\times \Omega)} = \|   \bm{\Psi} \cdot \bm{\Xi}  - \bm{\Psi}_{\text{NN}}\cdot \bm{\Xi}_{\text{NN}} \|_{L^2(D\times \Omega)} \\
\leq & \| \bm{\Psi} \cdot \bm{\Xi}   - \bm{\Psi}_{\text{NN}}\cdot \bm{\Xi}\|_{L^2(D\times \Omega)} + \| \bm{\Psi}_{\text{NN}}\cdot \bm{\Xi}   - \bm{\Psi}_{\text{NN}}\cdot \bm{\Xi}_{\text{NN}}\|_{L^2(D\times \Omega)}  \\
\leq & \|  \bm{\Psi} -  \bm{\Psi}_{\text{NN}} \|_{L^2(D;\R^N)}\| \bm{\Xi} \|_{L^2(\Omega;\R^N)} + \|  \bm{\Xi} -  \bm{\Xi}_{\text{NN}} \|_{L^2(D;\R^N)}\| \bm{\Psi}_{\text{NN}} \|_{L^2(\Omega;\R^N)}
\end{split}
\end{equation*}
Notice that $\| \bm{\Xi} \|_{L^2(\Omega;\R^N)}^2 = \sum_{j=1}^N \lambda_j$ by \cref{eq:truncation_error} and $\| \bm{\Psi}\|_{L^2(\Omega;\R^N)}^2 = N$, we have 
\begin{equation*}
\begin{split}
&\|  \bm{\Psi} -  \bm{\Psi}_{\text{NN}} \|_{L^2(D;\R^N)}\| \bm{\Xi} \|_{L^2(\Omega;\R^N)} + \|  \bm{\Xi} -  \bm{\Xi}_{\text{NN}} \|_{L^2(D;\R^N)}\| \bm{\Psi}_{\text{NN}} \|_{L^2(\Omega;\R^N)} \\
\leq &
\frac{\ep}{P} \sqrt{\sum_{j=1}^N \lambda_j} + \frac{\ep}{P} (\sqrt{N} + \frac{\ep}{P})<\ep/2. 
\end{split}
\end{equation*}
Therefore we only need to find the condition for \cref{eq:NNapproximation} to be satisfied. 
Notice that $M$ needs to satisfy $ \sqrt{\sum_{j=1}^N \lambda_j}/P+ \sqrt{N}/P + 1/P^2 < 1/2$. Hence one can choose $P$ such that $P = O(\sqrt{N})$. Now since we need $P$ such that
 $C_1 P^{-\frac{\lfloor\alpha\rfloor - n}{2n}}< \ep /2$, we have 
\[
\frac{\epsilon}{\sqrt{N}P}=O(\epsilon^{\frac{\lfloor\alpha\rfloor+n}{\lfloor\alpha\rfloor-n}}).
\]
%At last, 

Finally, by \Cref{lem:NNApproximation},  the number of weights in the solution-network is 
bounded above by 
\[
c_1 N (\ep^{\frac{\lfloor\alpha\rfloor+n}{\lfloor\alpha\rfloor-n}})^{-n/\alpha}
=c_1 \ep ^{-\frac{n(2\alpha+\lfloor\alpha\rfloor+n)}{\alpha(\lfloor\alpha\rfloor-n)}},
\]
where the constant $c_1 > 0$ depends only on $n$, $\beta$, and $\sig^{-1}\circ\Phi$, 
and the number of weights in the parameter-network is bounded above by
\[
c_2 N (\ep^{\frac{\lfloor\alpha\rfloor+n}{\lfloor\alpha\rfloor-n}})^{-n/\beta}
=c_2 \ep ^{-\frac{n(2\beta+\lfloor\alpha\rfloor+n)}{\beta(\lfloor\alpha\rfloor-n)}},
\]
where the constant $c_2 > 0$ depends only on $m$, $\beta$, and $\sig^{-1}\circ\Phi$. 

Combing the above results, we have $\| \sig^{-1}\circ \Phi - \hat{\Phi}_{\rm PSNN} \|_{L^2(D\times \Om)} < \ep$. Notice that the scaled sigmoid function \eqref{eq:scaledsigmoid}
is Lipschitz continuous with a Lipschitz constant less than $1$ for any $\eta<1$. Therefore, we conclude that 
\[
\| \Phi - \sig(\hat{\Phi}_{\rm PSNN}) \|_{L^2(D\times \Om)} \leq \| \sig^{-1}\circ \Phi - \hat{\Phi}_{\rm PSNN} \|_{L^2(D\times \Om)} < \ep,
\]
completing the proof.
%\YZ{$c_2 \ep ^{-\frac{2(k_1+n)(m-1)}{(k_1-n)(k+\alpha)}-\frac{2n}{k_1-n}}$}.
%If $\Phi$ is analytic, we should take $N$ satisfying $\beta(N) e^{-C_2N^{1/n}}< \epsilon/2$, where is $\beta(N)$ a polynomial of $N$ with degree $n$.
\end{proof}

%%%%%%%%%%%%%%%%%%%%%%%%%%%%%%%%%%%%%%%%%%%%%%%%%%%%%%%%%%%%%
\section{Numerical Results}
\label{sec:Results}

We use the Gray--Scott model \cite{GrayScott83,GrayScott84,GrayScott85} 
as an example to illustrate how our parameter-solution neural
network (PSNN) can be used to identify parameters with steady-state solutions and determine 
their stability.
%This model is a system of two ordinary differential equations (ODEs) 
%for two unknown functions, with two parameters. 
After reviewing the properties of steady-state solutions and their stability of the model, 
we shall test the convergence of our PSNN and algorithms applied to 
the Gray--Scott model. 
%%for approximating the parameter-solution neural network. 
Numerical phase diagrams will then be presented to show the boundaries
that separate the regions $\Omega_i$ of parameters with $\cN_i$ solutions
(cf.\ section~\ref{subsec:formulation} for the notation $\Omega_i$ and $\cN_i$), and 
the comparison between our numerical results and the 
known analytic results will also be shown. Finally, we report
some numerical results on training the PSNN using incomplete data. 

%%%%%%%%%%%%%%%%%%%%%%%%%%%%%%%%%%%%%%%%%%%%%%%%%%%%%%%
\subsection{The Gray--Scott model}
\label{ss:GS}

%In our paper, we use Gray-Scott model as an example to illustrate how coupled neural networks can help us %find the steady-states and the stability of the steady-states.

The full Gray--Scott model is a system of two reaction-diffusion equations of 
two unknown functions with two parameters. 
Relevant to our studies is the spatially homogeneous and steady-state part of 
the model, which for convenience is called the Gray--Scott model here: 
\beq
\label{eq:GrayScott}
\left\{
\begin{aligned}
     - u v^2 + f (1-u) &= 0, \\
      u v^2 - (f + k) v &= 0 .
\end{aligned}      
\right.
\eeq
Referring to the notations in \Cref{subsec:formulation}, here we have 
 $n=m=2$, $U = (u, v)$ and $\Theta = (f,k)$. 
 We define the solution space and the parameter space to be 
 $D = (0, 1) \times (0, 1) $ and $ \Omega = (0, 0.3) 
 \times (0, 0.08),$
 %\[
 %D = (0, 1)\times (0, 1) \quad \mbox{and} \quad  \Omega = (0, %0.3)\times (0,0.08), 
 %\]
 respectively \cite{GrayScott83,GrayScott84,GrayScott85}.
 One can note the Gray--Scott model has a trivial solution $(u, v) = (1, 0)$. In the upcoming discussion, we will exclude the trivial solution and concentrate solely on the nontrivial solutions.
 %Note that the trivial solution $(u, v) = (1, 0)$ of \eqref{eq:GrayScott} is not in
 %the solution space $D,$ and will therefore be ignored in our numerical tests. 
 %%Using our definition in section~\ref{subsec:formulation}, 

 Based on the quantity of solutions, the parameter space has the decomposition 
 $\overline{\Omega} = \overline{\Omega_0} \cup \overline{\Omega_1}, $  where 
 \[
 \Omega_0 = \{ (f, k) \in \Omega: f < 4 (f + k)^2 \}  \qquad \mbox{and} \qquad 
 \Omega_1 = \{ (f, k) \in \Omega: f > 4 (f + k )^2\}
 \]
 are disjoint open subsets of $\Omega$ such that 
 \begin{itemize}
 \item 
 for any parameter $\Theta = (f, k)\in \Omega_0$, \cref{eq:GrayScott} has no solution in $D, $ and 
 \item
 for any parameter $\Theta = (f, k) \in \Omega_1$, \cref{eq:GrayScott} has two distinct solutions $\hat{U}_1^\Theta$ and 
 $\hat{U}_2^\Theta$ in $D:$ 
%Decompose $\Om$ into $\Omega_1 =\Om\cap \{ f < 4 (f + k)^2\}$, and $\Omega_2 = \Om\cap \{ f > 4 (f + k)^2\}$ %such that $\cN_1 = 0$ and $\cN_2 = 2$. For any given $\Theta = (f, k)\in \Omega_2$,
\begin{equation}
\label{GSsolns}
\left\{
\begin{aligned}
%%  & \hat{U}_1^\Theta = (\hat{u}_1, \hat{v}_1)  =\left(\frac{2(f+k)^2}{f+ \sqrt{f^2-4f(f+k)^2}},\frac{f+ \sqrt{f^2-4f(f+k)^2}}{2(f+k)}\right) \in D, \\
     %% & \hat{U}_2^\Theta = (\hat{u}_2, \hat{v}_2)  = \left(\frac{2(f+k)^2}{f- \sqrt{f^2-
     %% 4f(f+k)^2}},\frac{f- \sqrt{f^2-4f(f+k)^2}}{2(f+k)} \right) \in D, 
     & \hat{U}_1^\Theta = (\hat{u}^\Theta_1, \hat{v}^\Theta_1)  =\left(
     \frac{ f  - \sqrt{f^2 - 4 f (f+k)^2} } { 2 f},\frac{f+ \sqrt{f^2-4f(f+k)^2}}{2(f+k)}\right), \\
     & \hat{U}_2^\Theta = (\hat{u}^\Theta_2, \hat{v}^\Theta_2)  = \left(
     \frac{ f + \sqrt{f^2 - 4 f (f+k)^2}}{2 f},\frac{f- \sqrt{f^2-4f(f+k)^2}}{2(f+k)} \right).
\end{aligned}      
\right.
\end{equation}
\end{itemize}
The two subsets $\Omega_0$ and $\Omega_1$ of $\Omega$ is separated by the 
parameter phase boundary for solution
\begin{equation}
    \label{solutionboundary}
    \Gamma_{\rm soln} = \{ (f, k ) \in \Omega: f = 4 (f + k)^2 \}.
\end{equation}
%%We shall call $\Gamma_{\rm soln}$ the {\it parameter phase boundary for solution}.  

The target function $\Phi$ defined in \eqref{eq:Phi} is now given by
\begin{align}
\label{PhiGS}
\Phi(U, \Theta) = 
\Phi((u,v), (f,k)) = 
\chi_{\Omega_1}((f,k))
\sum_{j=1}^{2} 
\exp{ \left( - \frac{| u - \hat{u}^\Theta_j|^2}{\delta((f,k))}-
\frac{| v - \hat{v}^\Theta_j|^2}{\delta((f,k))} \right)}.
\nonumber 
\end{align}
Considering further the stability of solutions, we have the decomposition 
$\overline{\Omega_1}  = \overline{\Omega_{1, 1}} \cup \overline{\Omega_{1, 2}}$, 
where $\Omega_{1, 1}$ and $\Omega_{1, 2}$ are two disjoint open subsets of $\Omega_1$, given by
\begin{align*}
    &\Omega_{1,1}= \{ (f, k) \in \Omega_1:  f\sqrt{f^2-4f(f+k)^2}+f^2-2(f+k)^3>0\}, \\
    & \Omega_{1,2}=\{ (f, k) \in \Omega_1:   f\sqrt{f^2-4f(f+k)^2}+f^2-2(f+k)^3<0\}.
\end{align*}
These open subsets are characterized by the following properties: 
\begin{itemize}
    \item 
    For each parameter $\Theta = (f, k) \in \Omega_{1,1}$, the solution 
$\hat{U}_1^\Theta$ in \eqref{GSsolns} is stable and the solution 
$\hat{U}_2^\Theta$  in \eqref{GSsolns} is unstable. 
\item
For each parameter $\Theta = (f, k) \in \Omega_{1, 2}$, both solutions
$\hat{U}_1^\Theta$ and $\hat{U}_2^\Theta$ in \eqref{GSsolns} are unstable.
\end{itemize}
These two subsets $\Omega_{1, 1}$ and $\Omega_{1, 2}$ of $\Omega_1$ are 
separated by the parameter phase boundary for stability 
\begin{equation}
    \label{stabilityboundary}
    \Gamma_{\rm stab} = \{ (f, k) \in \Omega_1:  f\sqrt{f^2-4f(f+k)^2}+f^2-2(f+k)^3=0\}.
\end{equation}
%%We shall call $\Gamma_{\rm stab}$ the {\it parameter phase boundary for stability}. 
%%%%%%%%%%%%%%%%%%%%%%
\begin{comment}
one of the solutions  there is a stable solution and an unstable solution; in $\Omega_{2,2}$, there are two unstable solutions. 
Indeed, the Jacobian matrix of the system \eqref{eq:GrayScott} is given by
\begin{displaymath}
    \mathbf{J}=\left( \begin{array}{cc}
  -v^2-f & -2uv \\
  v^2 & 2uv-(f+k) \\
   \end{array} \right)
\end{displaymath}
which gives us two eigenvalues $\lambda_1, \lambda_2$ satisfying
\[
\left\{
\begin{aligned}
  & \lambda_1 +\lambda_2=k-v^2, \\
  & \lambda_1 \lambda_2=(f+k)(v^2-f).
\end{aligned}      
\right.
\]     
For the steady state $(\hat{u}_2, \hat{v}_2)$, we have $\lambda_1 \lambda_2=(f+k)(v^2-f)<0$, which leads to a positive eigenvalue and a negative eigenvalue. So $(\hat{u}_2, \hat{v}_2)$ is an unstable solution. On the other hand, for $(\hat{u}_1, \hat{v}_1)$, $v^2-f>0$ is always guaranteed. When $f\sqrt{f^2-4f(f+k)^2}+f^2-2(f+k)^3>0$ holds, we have $\lambda_1 +\lambda_2=k-v^2<0$, which leads to two negative eigenvalues. In this case, $(\hat{u}_1, \hat{v}_1)$ is stable. Otherwise, $(\hat{u}_1, \hat{v}_1)$ is unstable. 
\end{comment}
%%%%%%%%%%%%%%%%%%%%%%%%%%%%%%%%%%
The target function for stability $\Phi^{\rm s}$ given in \eqref{eq:Phis} is now given by 
\begin{align*}
&\Phi^{\rm s} (U, \Theta)  = \Phi^{\rm s}((u,v), (f,k))\\
&= \chi_{\Omega_{1,1}} \bigg(\exp\left(- \frac{| u - \hat{u}^\Theta_1|^2}{\delta((f, k))}-\frac{| v - \hat{v}^\Theta_1|^2}{\delta((f, k))}\right)-\exp\left(- \frac{| u - \hat{u}^\Theta_2|^2}{\delta((f, k))}-\frac{| v - \hat{v}^\Theta_2|^2}{\delta((f, k))}\right) \bigg)\\
&-\chi_{\Omega_{1,2}} \bigg(\exp\left(- \frac{| u - \hat{u}^\Theta_1|^2}{\delta((f, k))}-\frac{| v - \hat{v}^\Theta_1|^2}{\delta((f, k))}\right)+\exp\left(- \frac{| u - \hat{u}^\Theta_2|^2}{\delta((f, k))}-\frac{| v - \hat{v}^\Theta_2|^2}{\delta((f, k))}\right) \bigg)
\\
&\qquad \qquad \qquad \qquad \forall (U, \Theta) = ((u, v), (f, k)) \in D \times \Omega.
\end{align*}
All these analytical properties of the Gray--Scott model
%%, solutions and their corresponding stability information can be obtained analytically 
can be used to generate observation data sets for training and testing our PSNNs. 
%%In practice, the observation data sets may be obtained from biology or chemistry experiments and our %%framework has the potential to tackle these situations as well.

%%%%%%%%%%%%%%%%%%%%%%%%%%%%%%%%%%%%%%%%%%%
\subsection{Convergence test}
\label{subsec: Convergence Test}

In this subsection, we conduct numerical experiments to verify the convergence of $\Phi_{\rm PSNN}$ to
the target function $\Phi$ defined in \eqref{eq:Phi_PSNN} and \eqref{eq:Phi}, respectively.
Following the analysis in \Cref{sec:Analysis}, one should observe such convergence with respect
to two aspects: the dimension $N$ of output vectors, and the architecture of two sub-networks $\Phi_{\rm PNN}$ and $\Phi_{\rm SNN}$, including the depth (i.e.,the number of layers) $L_1, L_2$ and the width $W_1, W_2$ (i.e., the number of neurons on each hidden layer).
\begin{comment}
\begin{itemize}
    \item 
    the dimension $N$ of output vectors from the two sub-networks $\Phi_{\rm PNN}$ and $\Phi_{\rm SNN}$ (cf.\ Figure~\ref{fig:NN});
\item 
the depth (i.e.,the number of layers) of each of the two sub-networks 
$\Phi_{\rm PNN}$ and $\Phi_{\rm SNN}$;
\item 
the width (i.e., the number of neurons) of hidden layers for 
each of the two sub-networks
$\Phi_{\rm PNN}$ and $\Phi_{\rm SNN}.$
\end{itemize}
\end{comment}
The convergence is measured by the decrease of the test error which is defined to be 
the mean-squared error between the outputs of the trained neural networks 
and the target function.
%and the increase of the sizes of these two networks. Here, the size of a neural network means 
%the width (i.e., the number of layers) and the number of weights in each layer of the network. 
%%verify the approximation power of the PSNN. Based on the theoretical results in \Cref{sec:Analysis},
%%two critical factors affect the approximation accuracy of the target function $\Phi$.  
%%One is the output vectors dimension (the dimension of output vector $1$ and $2$ in \Cref{fig:NN}), 
%%and another is the sizes of the parameter-network and the solution-network (
%%depicted in \Cref{fig:NN}).  

In the following numerical experiments, the size of the training data set $\cT_{\rm train}$, defined in \cref{Ttrain}, is taken with
$|I_{\rm train}| = 1000 $
and $ N_{\rm random} = 200$, and is fixed through all the experiments. The test dataset, denoted as $\cT_{\rm test}$ and generated using the same method as the training data, is fixed with $|I_{\rm test}| = 600$ and $N_{\rm random} = 200$. 
%And for the set up of target function $\Phi$ \cref{eq:Phi}, $\delta$ is chosen as $0.3$ to separate the two solutions. 
%and $C_{\Omega_1}$ is taken as $\max_{(\Theta, \{ \hat{U}_j^\Theta \}_{j=1}^2) \in \mathcal{T_{\rm observ}}} \{1+ \exp{ \left( - \frac{| \hat{u}^\Theta_1 - \hat{u}^\Theta_2|^2}{\delta}-
%\frac{| \hat{v}^\Theta_1 - \hat{v}^\Theta_2|^2}{\delta} \right)}\}.$

We first test the convergence with respect to the increase of the dimension of output vectors $N$. At the same time, we enlarge the depth of two sub-networks $L_1, L_2$ together, while the width $W_1, W_2$ are fixed as $30$ and $20$ respectively which are considered to be large enough. As shown in \Cref{fig:convergenceN}, we have $N$ increase from $2$ to $8$, and $L_1, L_2$ change together from $1$ to $6$.
%%hobserve the convergence concerning the dimension of the output vectors, denoted by $N$, given that %%the parameter-network and the solution-network are large enough. 
And by comparing different curves for each fixed depth, one can observe a clear mind overall decrease of the error as the value of $N$ increases. Also, focusing on each curve, one can find an obvious decrease of error as $L_1, L_2$ increase.
%%\Cref{fig:convergenceN} shows an obvious decrease in the test errors as $N$ increases from $2$ to $8$. 
\begin{figure}[htbp]
    \centering
    \includegraphics[scale=0.55]{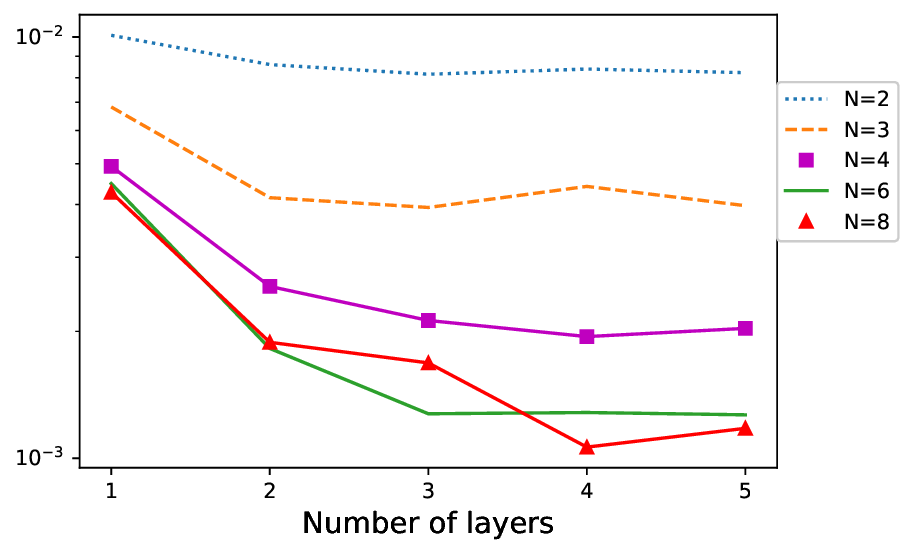}
    \caption{\footnotesize Test errors against the depth of the two
    sub-networks $\Phi_{\rm PNN}$ and $\Phi_{\rm SNN}$ for PSNN for different output dimensions $N$. Here the $x$-axis is the value of $L_1, L_2$ (with $L_1=L_2$). The $y$-axis gives the test error in $log$ scale. And curves in different colors depict the test error corresponding to different values of $N$.
    }
    \label{fig:convergenceN}
\end{figure}

We next test the convergence with respect to PSNN structures while fixing the dimension $N$. From the previous numerical tests, it is sufficient to take $N=8$. 
%%The theoretical analysis in 
%%\Cref{sec:Analysis} indicates that the  sizes of the parameter-network  
%%and the solution-network depend on the regularity of the function $\Phi(U, \Theta)$ in   
%%$\Theta$ and $U$, respectively. 
%%Here we test convergence with respect to the size change of each sub-network 
%%by fixing one network and changing the other. 
\Cref{fig:convergencesize} summarizes our convergence tests for the network
$\Phi_{\rm PNN}$ and $\Phi_{\rm SNN}$ separately. In \Cref{fig:convergencesize}(A), we plot the error against depth of $\Phi_{\rm SNN}$ for several values of the width of $\Phi_{\rm SNN},$ while 
the structure of $\Phi_{\rm PNN}$ is fixed with $4$ layers and $60$ 
neurons in each layer.
%%while the solution-network has varied structures for the convergence test. 
In contrast, \Cref{fig:convergencesize}(B) shows the test error against 
the depth for varied number of width of $\Phi_{\rm PNN}$, while the structure of $\Phi_{\rm SNN}$ is fixed 
with $3$ layers and $20$ neurons in each layer. By comparing the different curves in \Cref{fig:convergencesize}, one can observe 
 the decrease in test erro as the sizes of $\Phi_{\rm PNN}$ and $\Phi_{\rm SNN}$ increase. Additionally, it is worth noticing that the error decreases much faster with respect to the size increase for the solution-network than the parameter-network. This observation is 
 consistent with our analysis presented in \Cref{sec:Analysis} based on the 
 different regularity of the target function $\Phi(U, \Theta)$ in its variables 
 $\Theta$ and $U$.

\begin{figure}[htbp]
    \centering 
    \begin{subfigure}{0.43\textwidth}
 %%   \begin{subfigure}{0.395\textwidth}
        \includegraphics[clip=true, trim=0 0 0 0, width=\linewidth]{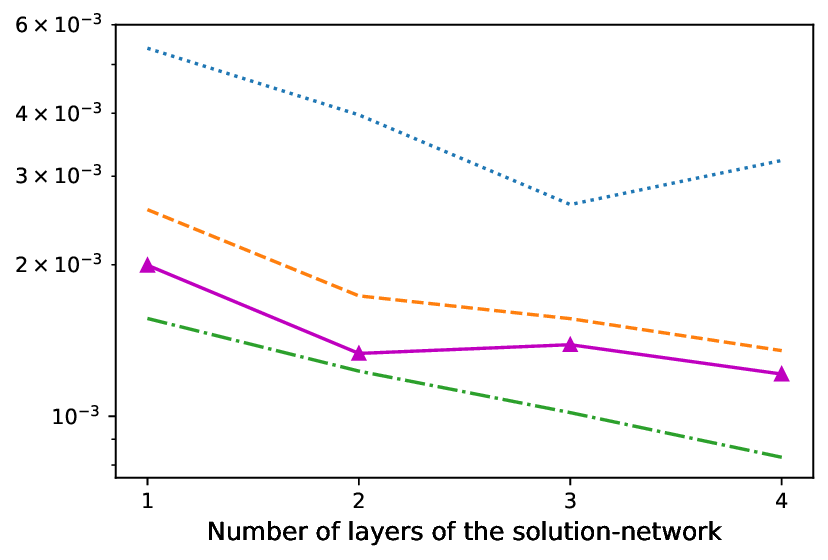}
        \caption{}
       % \label{fig:convergencesizeA}
    \end{subfigure}  
    \begin{subfigure}{0.53\textwidth}
%%    \begin{subfigure}{0.48\textwidth}
       \includegraphics[clip=true, trim=0 0 0 0, width=\linewidth]{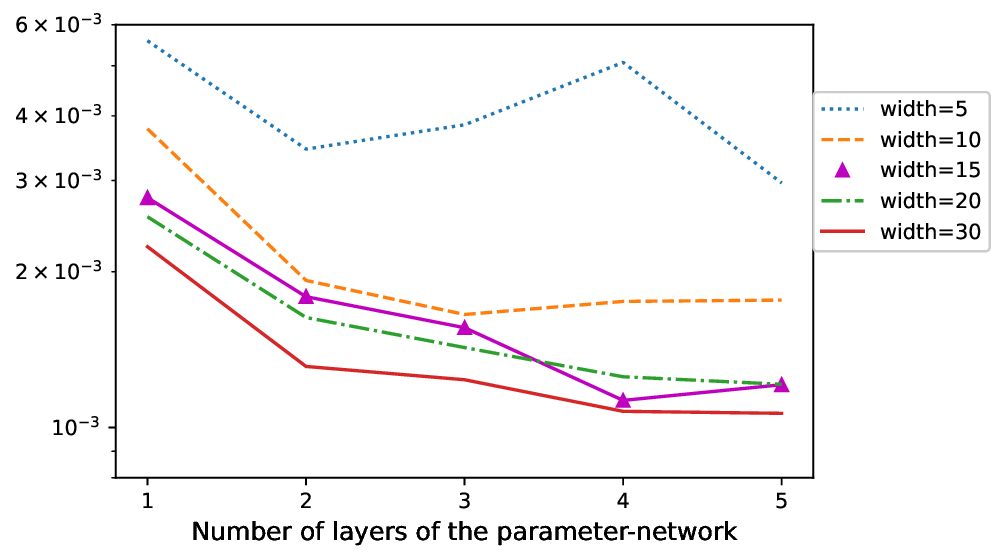}
        \caption{}
%\label{fig:convergencesizeB}
    \end{subfigure}
    \vspace{-10 pt}
    \caption{\footnotesize
    Test errors against the depth with varying width for the PSNN with a fixed dimension 
    $N = 8$ of the output vectors from the two sub-networks. The $x$-axis represents the value of $L_1$ (or $L_2$), and the $y$-axis gives the test error in $log$ scale.
    %%under different sizes (widths $\times$ \#layers).
    %%Output vector dimension $N=8$.
     (a) $\Phi_{\rm SNN}$ is tested for $L_2=1,2,3,4$ and $W_2=5, 10, 15, 20,$ and $\Phi_{\rm PNN}$ has the fixed structure: $L_1=4, W_1=30.$
     (b) $\Phi_{\rm PNN}$ is tested for $L_2=1,2,3,4,5$ and $W_2=5, 10, 15, 20, 30,$ and $\Phi_{\rm SNN}$ has the fixed structure: $L_1=3, W_1=20.$
     }
   \label{fig:convergencesize}
   \vspace{-20 pt}
\end{figure}

%%%%%%%%%%%%%%%%%%%%%%%%%%%%%%%%%%%%%%%%%%%%%%%%%%%%%%%%%%%%
%\subsection{Solution prediction and phase diagram}
%%\subsection{Locating solutions and determining phase diagram}
\subsection{Locating solutions and phase boundaries}
\label{subsec: phasediag}
%%\BL{I suggest to change the title of this subsection: Locating solutions and determining
%%phase boundaries. ``Locating solution" is the title of a subsection before; so consistent.}
We now present numerical results to show how our trained PSNN and the post-processing algorithm, 
Algorithm~\ref{alg:cluster}, can help determine the number of multiple solutions, locate such 
solutions approximately, and predict the stability of such solutions. We also show the error between
the predicted solutions and the exact solutions. 
%%In this subsection, we implement the algorithm proposed in \Cref{subsec:solutions} 
%%to locate solutions based on the learned PSNN. 
%%We are concerned with whether the algorithm can predict the 
%%correct number of steady-state solutions, 
%%the error between the predicted solutions and the exact solutions, and whether the 
%%algorithm can predict the stability of the solution correctly.
The prediction of the number of solutions and their stability properties
corresponding to the given parameters is
described by phase diagrams, which consist of phase boundaries in 
the parameter space $\Omega$. 
For the Gray--Scott model, the exact parameter phase boundary for solution and
the parameter phase boundary for stability are given explicitly in 
\eqref{solutionboundary} and \eqref{stabilityboundary}, respectively. 

We take the training data set and the test data set same as those 
%%in the way same as 
described in \Cref{subsec: Convergence Test}. 
The dimensions of the parameter network are set to $L_1=4$ and $W_1=30$, while the dimensions of the solution network are $L_2=3$ and $W_2=20$. The dimension $N$ of the output vectors 
from these sub-networks is set to be $N = 8.$ The maximum allowable number of solutions, denoted as $C_{\rm max}$, is fixed at 5; that is, the algorithms will check the range from 0 to 5 and then determine the number of solutions each parameter pair is associated with. 

The phase diagrams and the average error of locating solutions of two given algorithms are presented as follows. We compare the PSNN-based algorithm (see Algorithm \ref{alg:cluster}) and the naive mean-shift-based algorithm (see Algorithm \ref{alg:meanshift}) using the same fixed training and test data sets. 
%%For the PSNN, we set the parameter-network with a fixed size $50\times 4$, 
%%and the solution-network is fixed with a size $20\times 2$. The dimension of 
%%the output vectors for the two-subnetworks is $N=8$.
The phase diagrams for solution prediction for different parameters are plotted in \Cref{fig:phasediag}. 
Different colors in these plots correspond to different numbers of predicted solutions.  It is noticeable that the PSNN-based algorithm generally provides accurate predictions for the number of solutions corresponding to each parameter (see \Cref{fig:phasediag}(A)). Therefore, the phase boundary can be predicted quite faithfully using the algorithm. The mean-shift-based algorithm (see \Cref{fig:phasediag}(B)), however, does a poor job of predicting the number of solutions, especially when the number of solutions is non-zero.

\begin{figure}[htbp]
    \begin{subfigure}{0.49\textwidth}
    \includegraphics[scale=0.41]{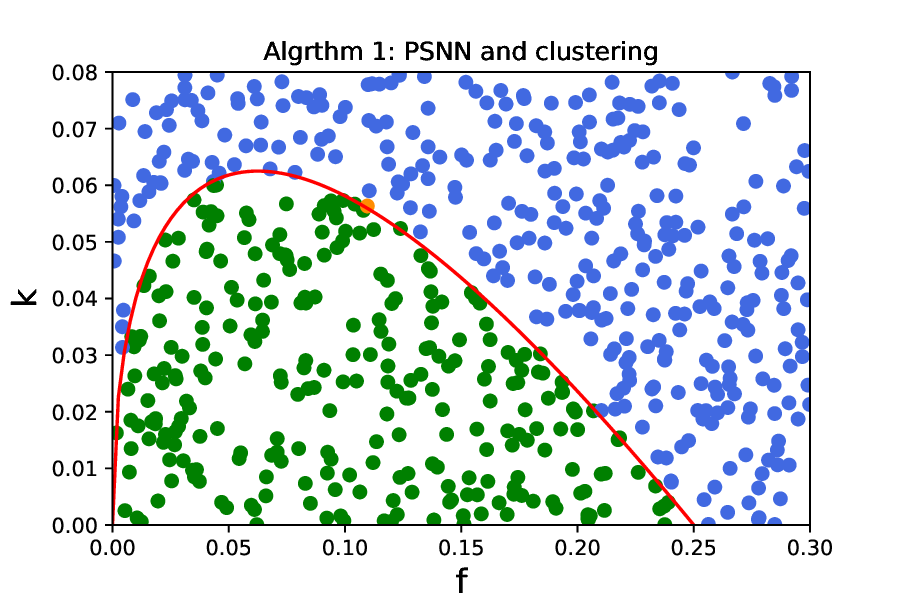}
        \caption{\footnotesize PSNN-based algorithm}
    \end{subfigure}  
    \begin{subfigure}{0.49\textwidth}
   \includegraphics[scale=0.41]{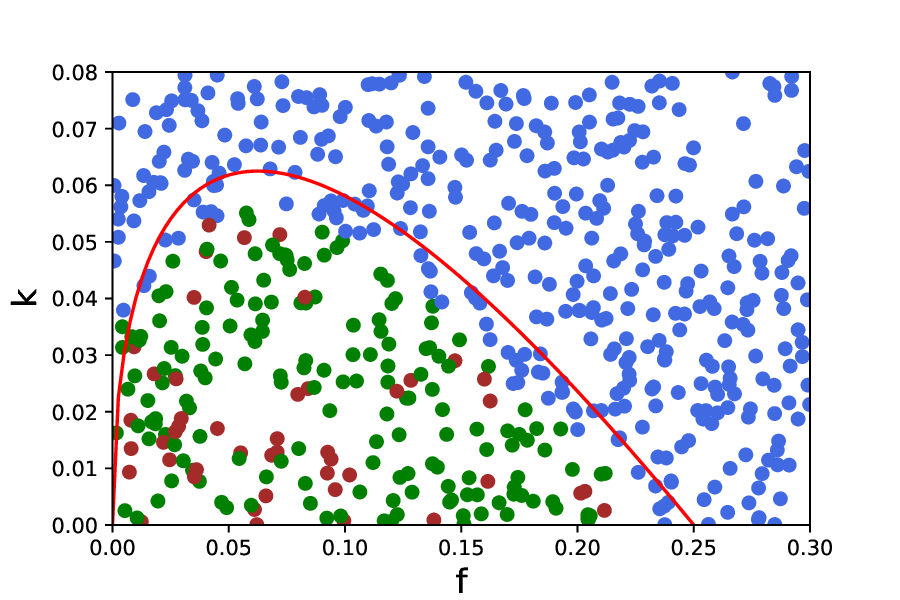}
        \caption{\footnotesize Mean-shift-based algorithm}
    \end{subfigure}
    \vspace{-10 pt}
    \caption{\footnotesize
    The phase diagram for solution prediction. The algorithms are set up to forecast the number of solutions within range $\{0,1, ,\dots, 5\}.$ The blue points represent the parameter pairs that the algorithm recognized as \textit{no-solution}, and the brown points, green points and orange points respectively correspond to \textit{1-solution}, \textit{2-solution}, \textit{3-and-more-solution}. The red curve represents  the phase boundary of $\Omega_0$ and $\Omega_1$, which is given by the formula $f-4(f+k)^2=0$.}
    \label{fig:phasediag}
\end{figure}
 
In addition, we use $\Phi^s_{\text{PSNN}}$ to predict the stability of the learned solutions and the corresponding phase diagram is shown in \Cref{fig:stability}.  An additional phase boundary curve is added that distinguishes stable solutions from unstable solutions in the region $\Om_2$. One can observe that the predicted result shows high consistency with the true phase boundaries.

\begin{figure}[htbp]
    \centering
    \vspace{-10 pt}
    \includegraphics[scale=0.45]{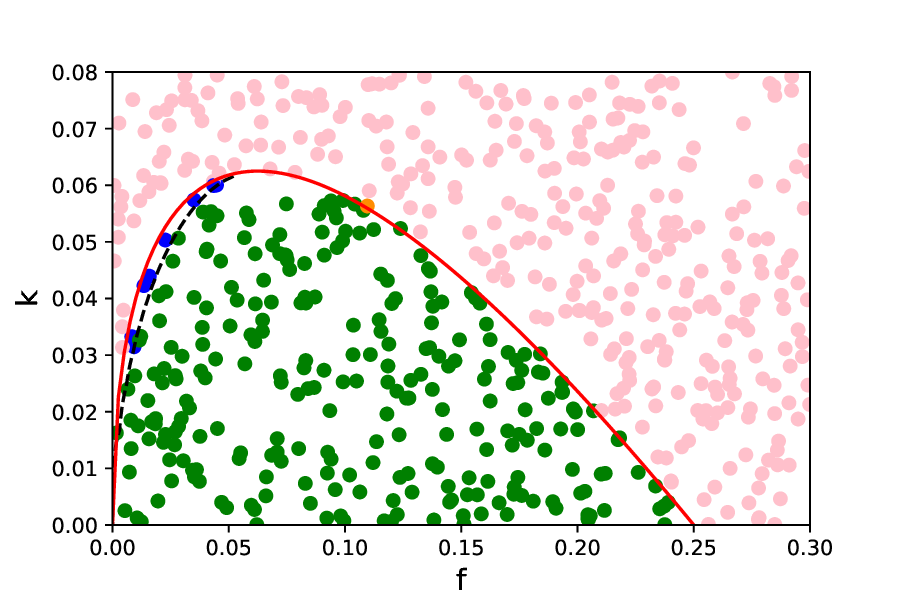}
    \vspace{-10 pt}
    \caption{\footnotesize  
    The Phase diagram for solution prediction with stability information. The algorithm is set to forecast the stability of learned solutions. The blue points, green points, brown points respectively represent \textit{2-unstable-solution}, \textit{1-stable-1-unstable-solution}, \textit{2-stable-solution}, and the black dashed curve plotted on $f\sqrt{f^2-4f(f+k)^2}+f^2-2(f+k)^3=0$ gives the boundary of $\Omega_{1,1}$ and $\Omega_{1,2}$.}
    \label{fig:stability}
    \vspace{-20 pt}
\end{figure}

%%%%%%%%%%%%%%%%%%%%%%%%%%%%%%%%%%%%%%%%%%%%%%%%%%%%%%%%
\subsection{Incomplete data}\label{subsec: incplt data}

In practical applications, observation data may not be complete, as defined
in \Cref{subsec:TrainingPSNN}. Here, we perform numerical tests on training our
neural networks with incomplete data sets to examine if our approach and algorithms can be
generalized.
%%there are occasions when not all the solutions are collected and the %observation data set may miss %%some solutions for a given parameter.
%To demonstrate the generalization power of our algorithms, we perform 
%%numerical tests on incomplete observation data sets. 
In this subsection, we take a training data set $\cT_{\rm train}$ with $|I_{\rm train}|=1200$ and $N_{\rm random}=200$ and the test data set is the same as described in \Cref{subsec: Convergence Test}. Among the parameters for which two solutions exist (approximately 510 out of the 1200 parameters), we take 120 of them and remove one of the two solutions for each in the observation data set to form an incomplete observation data set.
Furthermore, we perform a concentrated sampling technique designed to be compatible with incomplete data sets. For a complete data set, one can select $N_{\rm random}$ sampled points within the vicinity of each observed solution. This is illustrated by the left picture in \Cref{fig:incompletedatapts}, where $200$ random points are sampled from two neighborhoods of radius $2\delta(\Theta)$ of the observed solutions for a given parameter $\Theta=(f,k)$. Note that $\delta(\Theta)$ is defined in \eqref{eq:delta}. 
With an incomplete data set, the information 
surrounding the missing solution is removed, and we only sample points from the vicinity of the remaining observed solutions. The right picture in \Cref{fig:incompletedatapts} illustrates that when only one solution is observed for the same parameter $\Theta=(f, k)$, we only sample $100$ random points from a neighborhood of radius $2\tilde{\delta}(\Theta)$ of the observed solution. The determination of $\tilde{\delta}(\Theta)$ is carried out as follows. For a given parameter $\Theta$, let $n(\Theta)$ denote the number of observed solutions in a given (incomplete) data set. Define $m_\ep(\Theta): = \max \{n(\bar{\Theta}): \bar{\Theta}\in \cN_\ep(\Theta) \}$, where $\cN_\ep(\Theta) \subset \Omega$ denotes a neighborhood of $\Theta$ dictated by a small number $\ep>0$. Then we define
\beq
\label{eq:delta_incomplete}
\tilde{\delta}(\Theta) = \text{mean} \{ \delta(\bar{\Theta}): \delta(\bar{\Theta}) = m_\ep(\Theta) \text{ for } \bar{\Theta} \in \cN_\ep(\Theta) \}.
\eeq
Note that in the above definition, for any $\bar{\Theta}\in \cN_\ep(\Theta)$,  $\delta(\bar{\Theta})$ is computed by the formula in \eqref{eq:delta} where $S^{\bar{\Theta}}$ is taken as the observed (incomplete) solution set.
 The objective is to have $\tilde{\delta}(\Theta)$ serve as an approximation to the true value $\delta(\Theta)$ computed with a complete observation of solutions. In our experiment, we take $\cN_\ep(\Theta)$ to be a neighborhood of $\Theta$ that contains a few numbers of other parameters to ensure a large likelihood that $\tilde{\delta}(\Theta)$ approximates the true value. 
 
\begin{figure}[t]
\begin{subfigure}{0.49\textwidth}
    \includegraphics[clip=true, trim=20 10 30 0, width=\linewidth]{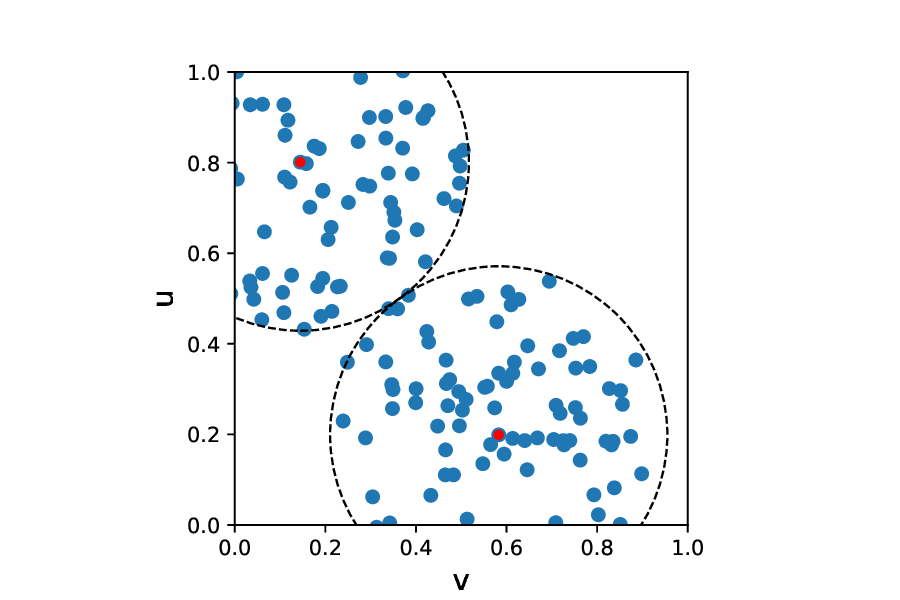}
\end{subfigure}
    \begin{subfigure}{0.49\textwidth}
     \includegraphics[clip=true, trim=20 10 30 0, width=\linewidth]{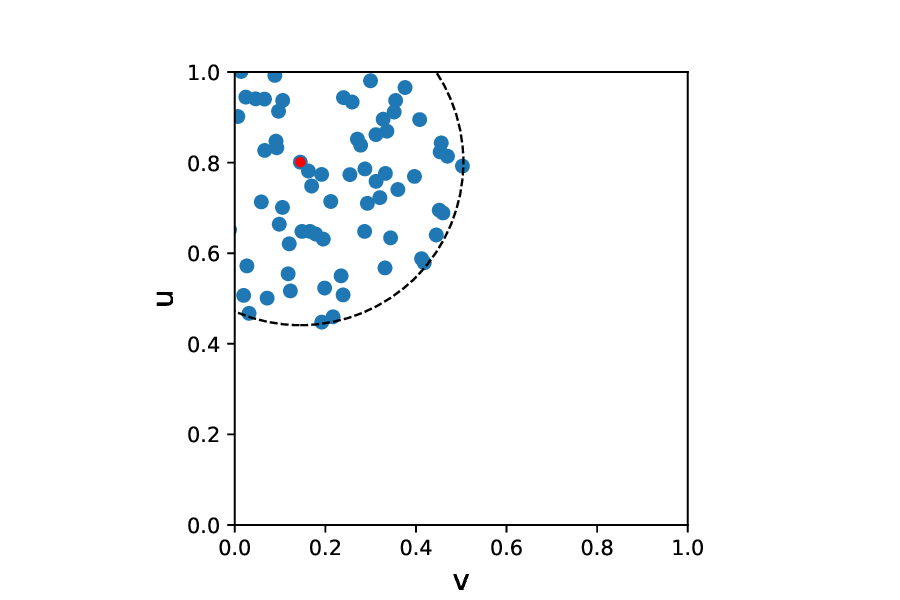}
\end{subfigure}  
        \caption{\footnotesize 
        Comparison of  complete data points ({\it left}) with incomplete data points ({\it right})  for a given pair $\Theta=(f,k)$ for which two solutions exist. Left: 200 randomly selected points in two neighborhoods of radius $2\delta(\Theta)$ of the observed solutions. Right: 100 randomly selected points with a  neighborhood of $2\tilde{\delta}(\Theta)$ the observed solution.
        $\delta(\Theta)$ and $\tilde{\delta}(\Theta)$ are defined in \cref{eq:delta,eq:delta_incomplete}, respectively. 
        }
        \label{fig:incompletedatapts}
         \vspace{-20 pt}
\end{figure}

\begin{figure}[htbp]
\vspace{-10 pt}
\begin{subfigure}{0.48\textwidth}
    \includegraphics[clip=true, trim=20 10 30 0, width=\linewidth]{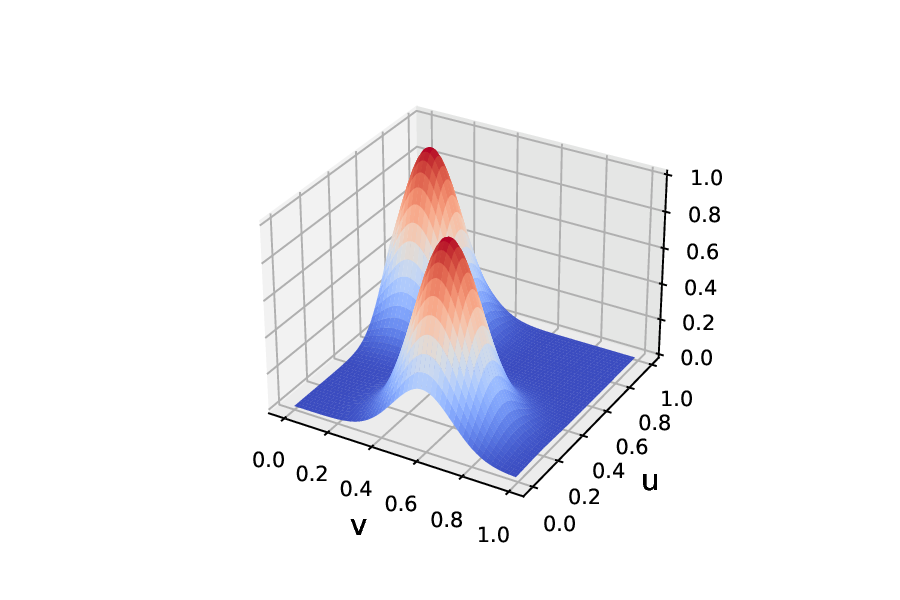}
\end{subfigure}
\begin{subfigure}{0.48\textwidth}
    \includegraphics[clip=true, trim=20 10 30 0, width=\linewidth]{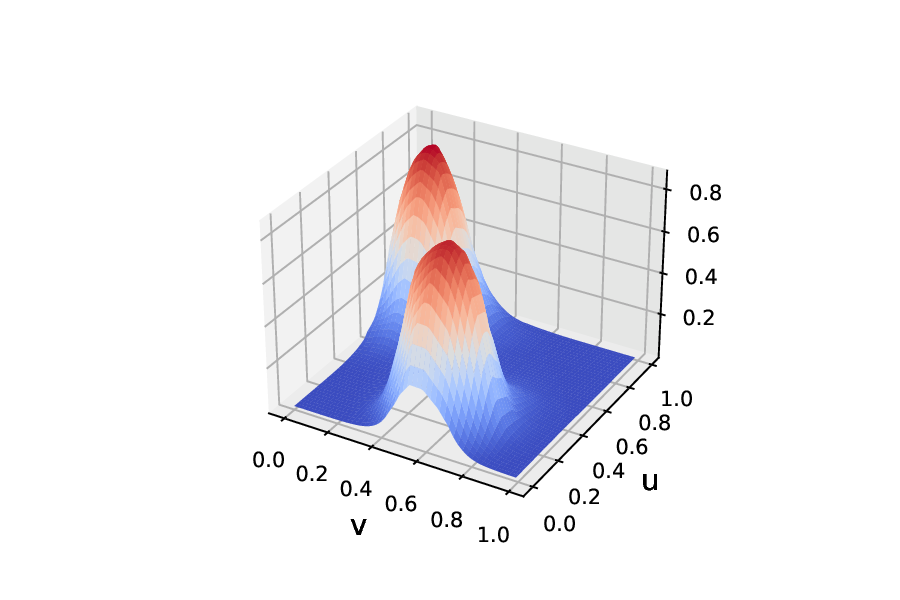}
\end{subfigure}
\vspace{-10 pt}
\caption{\footnotesize
 With the same parameter pair $\Theta=(f,k)$ in \Cref{fig:incompletedatapts}, the left figure shows the target function $\Phi((u, v), (f,k))$, and the right figure shows the PSNN-recovered function $\Phi_{\text{PSNN}}((u, v), (f,k))$ with the incomplete complete training data. PSNN recovers the missing information surrounding one of the solutions.
        }
        \label{fig:RBF}
         \vspace{-10 pt}
\end{figure}

We observe that the trained PSNN on incomplete data successfully covers the missing information, as illustrated by \Cref{fig:RBF}.
The left plot in \Cref{fig:RBF} shows the target function $\Phi$ for the same parameter $\Theta$ chosen in \Cref{fig:incompletedatapts}. In particular, the training data set has no information around the missing solution corresponding to the given parameter (cf. \Cref{fig:incompletedatapts}(right)). Yet, as shown by the right plot in \Cref{fig:RBF}, the trained PSNN successfully predicted the function surrounding the missing solution.

Additionally, we use the trained PSNN on incomplete data for solution prediction and draw the phase diagram in \Cref{fig:phasediag_incomplete} in comparison with the one predicted by the mean-shift-based algorithm. With incomplete data, the PSNN-based algorithm again successfully predicts the phase diagram, similar to the one in \Cref{fig:phasediag}.
In \Cref{tab:error_table}, quantitative errors for solution prediction on both complete and incomplete data sets are presented. For the experiment with incomplete data, in addition to the test on the randomly chosen test data set, we also record the algorithm prediction for the $120$ parameters with missing information in the observation data.  PSNN demonstrates effectiveness cross all tests.  
%In \Cref{tab:error_table}, it's more clear to see how the PSNN-based algorithm significantly outperforms the meanshift algorithm in predicting the number of solutions, achieving a high accuracy of 0.992. Additionally, the solutions located by PSNN-based algorithm demonstrate remarkable proximity to the true solutions,  with an average relative distance of 0.0017. Furthermore, despite being trained on incomplete data, the PSNN still generates the algorithm successfully, exhibiting an accuracy of 0.0086 in forecasting the number of solutions on random test set and zero-free recovery of the solutions within the lost data set.
As a comparison, the naive mean-shift-based algorithm exhibits significantly poorer performance than the PSNN-based algorithm, particularly on the lost data set. 
%It's noteworthy that we select Meanshift algorithm as comparison to highlight the pivot role of our PSNN. This choice is based on the fact that the same data points of the same kernel functions are being employed in Meanshift algorithm as in training the PSNN. By comparison, one can find the remarkable effectiveness of PSNN in predicting the solutions for random unknown parameters. And this impressive capability, indicated by our analysis, is attributed to the precise interpolation of the target function achieved by our parameter-solution neural networks.

\begin{table}[hptb]
    \centering
    \begin{tabular}{|c|c|c|c|c|c|c|}
    \hline
    \multirow{2}{*}{ Training set}&\multirow{2}{*}{ Test set}&\multicolumn{3}{c|}{  PSNN}&\multicolumn{2}{c|}{ Mean-shift}\\
    \cline{3-5}  \cline{6-7}
    & & \scriptsize Wrong-soln& \small Distance&\scriptsize Wrong-stb&\scriptsize Wrong-soln&\small Distance\\
    \hline
    Complete&Random&1.22\%&0.020&5.12\%&14.06\%&0.149\\
    \hline
    \multirow{2}{*}{Incomplete}&Random&1.22\%&0.020&4.58\%&14.22\%&0.149\\
    \cline{2-7}
    &Lost data&0.56\% &0.018&0\%&36.94\%&0.147\\
    \hline
    \end{tabular}
    \caption{ \footnotesize Error table of PSNN-based algorithm and mean-shift-based algorithm using the complete/incomplete data, and performed on random/lost test data.  \textit{Wrong-soln} counts the percentage of parameters for which the algorithm fails to determine the correct number of solutions. \textit{Distance} signifies the average relative distance (as defined in \eqref{def:relativedist}) between the predicted solution set and the exact solution set of the parameters where the number of solutions has been correctly predicted. And \textit{Wrong-stb} denotes the percentage of parameters for which the algorithm's predictions on stability were inaccurate. To reduce the impact of randomness, we performed each set of experiments three times and evaluated the two algorithms on three separate test sets, calculating the average results.} 

    \label{tab:error_table}
\end{table}
\begin{figure}[htbp]
\vspace{-10 pt}
    \begin{subfigure}{0.45\textwidth}
\includegraphics[clip=true, trim=0 0 30 0, width=\linewidth]{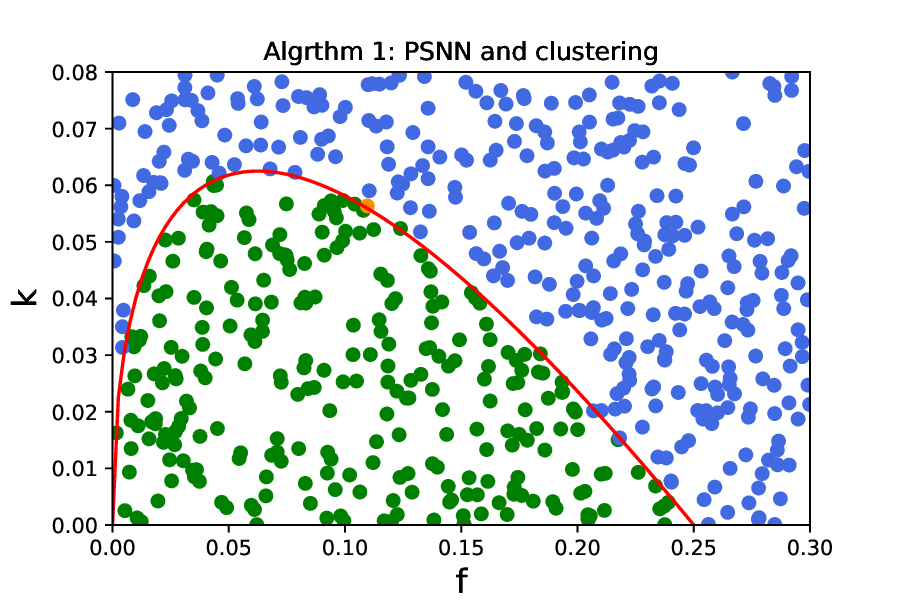}
        \caption{\footnotesize PSNN-based algorithm with incomplete data}
    \end{subfigure}  
    \begin{subfigure}{0.45\textwidth}
  \includegraphics[clip=true, trim=0 0 30 0, width=\linewidth]{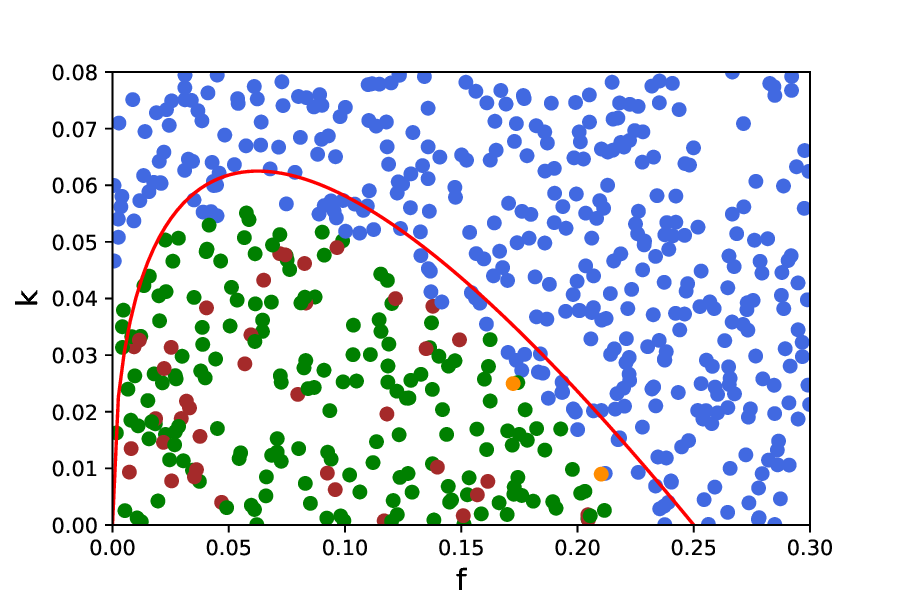}
        \caption{\footnotesize Mean-shift-based algorithm with incomplete data}
    \end{subfigure}
    \caption{\footnotesize 
    Phase diagrams for solution prediction using incomplete data. Similar to \cref{fig:phasediag}, the blue points, the brown points, green points and orange points respectively represent the parameter pairs that the algorithm recognized as \textit{no-solution}  \textit{1-solution}, \textit{2-solution}, \textit{3-and-more-solution}, and the red curve represents the phase boundary of $\Omega_0$ and $\Omega_1$. }
    \label{fig:phasediag_incomplete}
\end{figure}

\begin{figure}[htbp]
    \centering
    \vspace{-.5cm}\includegraphics[scale=0.45]{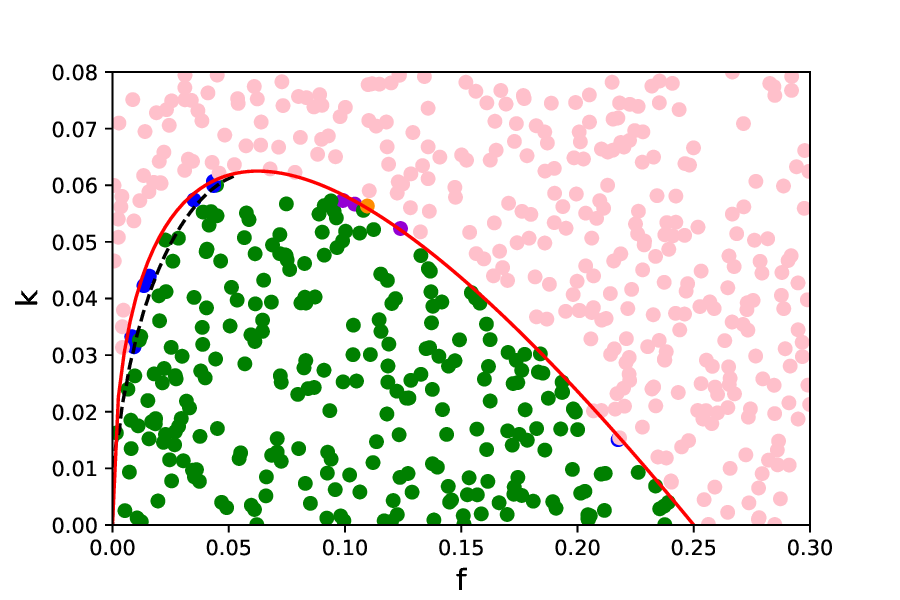}
    \vspace{-5 pt}
    \caption{\footnotesize  
    The hase diagram for solution prediction with stability information using incomplete data. 
    The blue points, green points, brown points respectively represent \textit{2-unstable-solution}, \textit{1-stable-1-unstable-solution}, \textit{2-stable-solution}, and the black dashed curve plotted on $f\sqrt{f^2-4f(f+k)^2}+f^2-2(f+k)^3=0$ gives the boundary of $\Omega_{1,1}$ and $\Omega_{1,2}$.}
    \label{fig:stability_incomplete}
    \vspace{-20 pt}
\end{figure}

Finally, we train $\Phi^s_{\rm PSNN}$ on the incomplete data set, and the phase diagram for stability is presented in 
\Cref{fig:stability_incomplete}. The above studies show that even trained with incomplete data, the PSNN-based algorithm continues to excel in forecasting the number of solutions for given parameters, predicting the phase boundary in parameter space, and determining the stability of learned solutions. 

%%\vspace{-2 mm}

%%%%%%%%%%%%%%%%%%%%%%%%%%%%%%%%%%%%%%%%%%%%%%%%%%%%%%%%%%%%%%%
\section{Conclusion}
\label{sec:Conclusions}

Systems of ordinary differential equations (ODEs) and dynamical systems with many parameters are widely used in scientific modeling with
emerging applications.  Identifying parameters for which solutions of such systems exist
and possess distinguished properties is a challenging task, particularly for complex and large
systems with many parameters. Here, we have initiated the development of a data-driven machine learning approach
to tackle such an important and difficult problem. Our approach is equation-free and applicable to finding solutions for general parameterized nonlinear systems.

We have first introduced target functions, $\Phi = \Phi(U, \Theta)$ and $\Phi^{\rm s}(U, \Theta),$
that characterize the desired parameter-solution properties, and then constructed 
parameter-solution neural
networks (PSNNs), $\Phi_{\rm PSNN}$ and $\Phi^{\rm s}_{\rm PSNN}$, for studying solutions and the stability of solutions, respectively. Each of these neural networks
couples a parameter-network and a solution-network with different structures, allowing us to treat
the parameters and solutions differently due to the different regularities of the
target functions for different
variables. We have also developed numerical methods to locate
solutions and determine their stability with our trained networks. 
We have presented a detailed analysis to show the convergence of our designed neural 
networks to the target functions with respect to the increase in network sizes, and have 
also obtained the related error estimates. 
Our extensive numerical results on the Gray--Scott model have confirmed our
convergence analysis and demonstrate the effectiveness of our approach in finding solutions, predicting phase boundaries, and assessing solution stability. 
In contrast to traditional techniques, our new approach, which is developed with the help 
of rigorous analysis, makes it possible to explore efficiently the 
entire solution space corresponding to each parameter, and further locate the multiple solutions
approximately and determine their solution stability. 

%%This novel network is both analyzed theoretically and tested numerically in our paper, and shown able %%to produce accurate approximation of the distribution. Moreover, we have developed a theory with %%quantitative estimates, which will indicate the expecting size of neural network according to equation %%system and the accuracy level we aim at. And as shown in the numerical results on 
%%the Gray--Scott model, our algorithm successfully predicts the number of solutions and 
%%the stability of each solution, and locates the solution and depicts the phase diagram precisely, %%showing great potential to solve for steady-states of reaction-diffusion models.

%%%%%%%%%%%%%%%%
\begin{comment}
To tackle such an important and difficult problem, in this work we initiate the development 
of a machine learning approach.  We consider systems of autonomous ODE with parameters, and 
have designed a novel architecture, a parameter-solution neural network (PSNN), which can predict the relation of parameters and variables, with the sub-network with different structures of the parameter and the variable permitted in the meantime. Based on the predicted distribution from network, we proposed our algorithm. 
\end{comment}
%%%%%%%%%%%%%%%%%%%

While our initial studies are promising, we would like to 
address several issues for the further development of our theory and numerical methods. 
First, we have used one target function $\Phi$ for solutions 
and another $\Phi^s$ for their stability. It may be desirable if these two can be combined to make the approach more efficient. 
Moreover, it demands further consideration and innovative approaches to manage large systems effectively. 
 For example, our post-processing method for locating solutions for a given parameter
$\Theta \in \Omega$ relies on a set of sample points
$\cU \subset D$ generated from uniform grids on $D$. For a high-dimensional problem, where the dimension 
of the solution vector is large, our method may not be feasible as the number
of data points in $\cU$ can be large. Alternative algorithms for locating solutions more efficiently
for high-dimensional problems need to be designed.  
A key question is the selection of a density from which to sample the points. Additionally, training networks with incomplete data is crucial in applications, especially for large
systems. Missing information in incomplete data, however, may be recovered from the 
structure of the underlying systems, as demonstrated by our initial examples in \Cref{subsec: incplt data}. 
Furthermore,  future work may explore methods that integrate equation-based and data-driven approaches.
Last but not least, applications of our framework to more realistic parameterized dynamic systems and/or parameterized nonlinear systems of algebraic equations will be considered in the future.

\appendix

%%%%%%%%%%%%%%%%%%%%%%%%%%%Appendix A%%%%%%%%%%%%%%%%%%%%%%%%%%

\section{Mean-shift algorithm}
%\renewcommand{\thesection}{A}
%\setcounter{equation}{0}
%section{}
\label{appendix}

%%%%%%%%%%%%%%%%%%%%
\begin{comment}
\BL{Need to describe the method briefly and then present the algorithm. Need to 
revise the algorithm using the same notations as in other algorithms. But need to 
decide first if we really want to include this algorith. The numerical results
on phase digrams from the PSNN clustering and mean-shift, it's hard to say that 
the mean-shift is worse. I in fact suggest to completely remove the mean-shift part.}
\end{comment}
%%%%%%%%%%%%%%%%%%%%

\begin{algorithm}[hptb]
\caption{Finding solutions by mean-shift and clustering}
\label{alg:meanshift}
\SetKwProg{Fn}{Function}{}{}
\SetKwFunction{Fdistance}{dist}
\SetKwFunction{Fneighbor}{neighbor}

\SetKwInput{Input}{Input}
\SetKwInput{Output}{Output}
\Input{A parameter $\Theta\in \Om$, a training set $\cT_{\rm train}$ (cf.\ \eqref{Ttrain}), 
%%given in section~\ref{subsec:TrainingPSNN}, 
a cut value $L_{\rm cut}\in (0,1),$, two numerical parameters
$\gamma_{\rm P} > 0$ and $\gamma_{\rm S}>0$ for defining neighbords in
$\Omega $ and $D$, respectively, and a tolerance $\ep_{\rm tol}$.} 
\Output{The set of centers $\cU^\Theta\subset S^\Theta$}

let $y_j^i = \Phi(U_j^i, \Theta_i)$ for $j=1,\cdots, S+\cN_{m_i}$, and $i=1,2,\cdots, P$ \;

%%select $\delta>0, \gamma>0$ to define the searching neighborhood;

%%define  \Fneighbor{$ \tilde{U}, \tilde{\Theta}$}:= $\{ (U_j^i, \Theta_i): \| \Theta_i -%%\tilde{\Theta}\|_{l^\infty}<\del, \| U^i_j - \tilde{U}\|_{l^\infty}<\gamma\}$\;

define  \Fneighbor{$ \tilde{U}, \tilde{\Theta}$}:= $\{ (U_j^i, \Theta_i): \| \Theta_i -\tilde{\Theta}\|_{l^\infty}<\gamma_{\rm P}, \| U^i_j - \tilde{U}\|_{l^\infty}<\gamma_{\rm S}\}$\;

\SetKwFunction{FMeanshift}{Meanshift}
\Fn{\FMeanshift{$\Theta,\gamma_{\rm P},\gamma_{\rm S}, \ep_{\rm tol}$}}{
Pick $U \in D$ randomly\;
$\ep \gets 1$\;
\text{condition} $\gets$ \textbf{True}\;
\While{$\ep \geq \ep_{\rm tol}$ and \text{condition}= \textbf{True}}{
   \eIf{$y_{\text{sum}}:=\underset{\Fneighbor{$U, \Theta$}}{\sum}y^i_j >0$}{
   $U^0 \gets U$ \;
   $U \gets \underset{\Fneighbor{$U, \Theta$}}{\sum}  \frac{y^i_j}{y_{\text{sum}}}  U_j^i$\;
   $\ep \gets$ \Fdistance{$U,U^0 $}
   }{
   \text{condition} $\gets$ \textbf{False}
   }
  }
  \If{condition}{
     \KwRet $U$
  }
}
 select $\ep_{\rm tol} \in (0,1)$\;
set $\cU_{\rm collected}$ to be an empty set\;
  \For{$1 \leq k \leq N_{initial}$}{
  $U_M \gets $ \FMeanshift{$\Theta,\gamma_{\rm P},\gamma_{\rm S}, \ep_{\rm step}$}

  \If{$\frac{1}{|\Fneighbor{$U_M, \Theta$}|}\underset{\Fneighbor{$U_M, \Theta$}}{\sum}y^i_j \geq L_{\rm cut}$}{
     Collect $U_M$ to $\mathcal{U}_{\rm collected}$
  }
}
  obtain $\cU_c \gets$ \FCluster{$\cU_{\rm collected}$} where \FCluster is given in Algorithm \ref{alg:cluster}\;
\end{algorithm}

%%%%%%%%%%%%%%%%%%%%%%%%Acknowledgement%%%%%%%%%%%%%%%%%%%%%%%%

\vspace{-5 mm}

%%%%%%%%%%%%%%%%%%%%%%%%%%%References%%%%%%%%%%%%%%%%%%%%%%%%%%

\bibliographystyle{siamplain}

\bibliography{ref}

\end{document}

%% file: shared.tex
\headers{Learning Steady States in Parameterized ODEs}{Yimeng Zhang, Alexander Cloninger, Bo Li, Xiaochuan Tian}

\title{A Neural Network Kernel Decomposition for Learning Multiple Steady States in Parameterized Dynamical Systems\thanks{Submitted to the editors DATE.
\funding{This work was partially supported by 
the NSF through the grant DMS-2111608 and DMS-2240180 (Y.Z. and X.T.), the grant DMS-2012266 (Y.Z. and A.C.), and the grant DMS-2208465 (B.L.), and a gift from Intel (A.C.).
}}}

\author{Yimeng Zhang\thanks{Department of Mathematics, University of California, San Diego, CA 92093, United States
(\email{yiz014@ucsd.edu}, \email{acloninger@ucsd.edu}, \email{bli@ucsd.edu}, \email{xctian@ucsd.edu}).}
\and Alexander Cloninger\footnotemark[2] \thanks{Halicioglu Data Science Institute, University of California, San Diego, Lo Jolla, CA 92093, United States} 
\and Bo Li\footnotemark[2] \thanks{Ph.D.\ Program in Quantitative Biology, University of California, San Diego, Lo Jolla, CA 92093, United States}
\and Xiaochuan Tian\footnotemark[2]}

\usepackage{multirow}
\usepackage{amsmath}
\usepackage{amssymb}
\usepackage{bm}

\usepackage{xcolor}
\usepackage{indentfirst}
\usepackage{graphicx}

\usepackage{etoolbox}
\usepackage{mathrsfs,esint}
\usepackage{comment}
\usepackage{enumerate}
\usepackage{hyperref}
\hypersetup{
    colorlinks = true,
    linkcolor = red,
    anchorcolor = green!50!black,
    citecolor = green!50!black,
    filecolor = green!50!black,
    urlcolor = green!50!black
}

\usepackage{float}
\usepackage[ruled, vlined]{algorithm2e}

\usepackage{subcaption}
\usepackage{pgfplots}
\pgfplotsset{compat=newest, compat/show suggested version=false}

\usepackage{tikz}
\usetikzlibrary{decorations.pathreplacing, decorations.pathmorphing, decorations.shapes}
\usetikzlibrary{arrows,calc}
\usetikzlibrary{intersections}
\usetikzlibrary{shapes.geometric}
\usepackage{etoolbox} 
\usepackage{listofitems}
\usepackage{tikz-cd}
\usepackage{tikz-3dplot}
\tdplotsetmaincoords{60}{20}
\tikzset{>=latex} 
\colorlet{myred}{red!80!black}
\colorlet{myblue}{blue!80!black}
\colorlet{mygreen}{green!60!black}
\colorlet{mydarkred}{myred!40!black}
\colorlet{mydarkblue}{myblue!40!black}
\colorlet{mydarkgreen}{mygreen!40!black}
\tikzstyle{node}=[very thick,circle,draw=myblue,minimum size=22,inner sep=0.5,outer sep=0.6]
\tikzstyle{connect}=[->,thick,mydarkblue,shorten >=1]
\tikzset{ % node styles, numbered for easy mapping with \nstyle
  node 1/.style={node,mydarkgreen,draw=mygreen,fill=mygreen!25},
  node 2/.style={node,mydarkblue,draw=myblue,fill=myblue!20},
  node 3/.style={node,mydarkred,draw=myred,fill=myred!20},
}
\def\nstyle{int(\lay<\Nnodlen?min(2,\lay):3)}

\newtheorem{thm}{Theorem}[section]

\newtheorem{remark}[thm]{\textit{Remark}}

\newcommand{\XT}[1]{{\color{orange} {\bf Xiaochuan:} [#1] }}

\newcommand{\BL}[1]{{\color{red} {\bf Bo:} [#1] }}

%% file: definitions.tex
\newcommand{\beq}{\begin{equation}}
\newcommand{\eeq}{\end{equation}}

\newcommand{\argmin}{\text{argmin}}

\def\cE{\mathcal{E}}

\def\cG{\mathcal{G}}
\def\cT{\mathcal{T}}

\def\cL{\mathcal{L}}

\def\cN{\mathcal{N}}
\def\cO{\mathcal{O}}

\def\cT{\mathcal{T}}
\def\cU{\mathcal{U}}

\def\N{\mathbb{N}}

\def\R{\mathbb{R}}

\def\Z{\mathbb{Z}}

\def\sD{\mathscr{D}}

\newcommand{\al}{\alpha}

\newcommand{\del}{\delta}
\newcommand{\ep}{\epsilon}

\newcommand{\sig}{\sigma}
\newcommand{\om}{\omega}

\newcommand{\Om}{\Omega}